\documentclass[leqno,12pt]{article} 
\usepackage{amsmath, amssymb}
\usepackage{amsthm} 
 \usepackage{stmaryrd}

 \allowdisplaybreaks[4]
\theoremstyle{plain} 
\newtheorem{theorem}{\indent\sc Theorem}[section]

\newtheorem{corollary}[theorem]{\indent\sc Corollary}
\newtheorem{prop}[theorem]{\indent\sc Proposition}

\theoremstyle{definition} 
\newtheorem{defn}[theorem]{\indent\sc Definition}
\newtheorem{rem}[theorem]{\indent\sc Remark}
\newtheorem{example}[theorem]{\indent\sc Example}

\makeatletter
\def\address#1#2{\begingroup
\noindent\parbox[t]{7.8cm}{%
\small{\scshape\ignorespaces#1}\par\vskip1ex
\noindent\small{\itshape E-mail address}%
\/: #2\par\vskip4ex}\hfill%
\endgroup}%
\makeatother
\title{\uppercase{Jacobi-Koszul-Vinberg structures on Jacobi-left-symmetric algebroids}} 
\author{
\bigskip \\
\textsc{Naoki Kimura and Tomoya Nakamura$^{*}$} 
}
\date{} 
\begin{document}

\maketitle

\footnote{ 
$^{*}$The second author was partially supported by JSPS KAKENHI Grant Number JP23K12977.
}

\begin{abstract}
A Koszul-Vinberg manifold is a generalization of a Hessian manifold, and their relation is similar to the relation between Poisson manifolds and symplectic manifolds.  Koszul-Vinberg structures and Poisson structures on manifolds extend to the structures on algebroids.  
A Koszul-Vinberg structure on a left-symmetric algebroid is regarded as a symmetric analogue of a Poisson structure on a Lie algebroid.  
We define a Jacobi-Koszul-Vinberg structure on a Jacobi-left-symmetric algebroid as a symmetric analogue of a Jacobi structure on a Jacobi algebroid.  We show some properties of Jacobi-Koszul-Vinberg structures on Jacobi algebroids and give the definition of a Jacobi-Koszul-Vinberg manifold, which is a symmetric analogue of a Jacobi manifold.


  \end{abstract}




\section*{Introduction}
A Poisson manifold is a generalization of a symplectic manifold and a non-degenerate Poisson manifold is equivalent to a symplectic manifold.  A Jacobi manifold, which is also a generalization of a contact manifold, is a further generalization of a Poisson manifold.  Poisson and Jacobi structures on manifolds extend to the structures on Lie and Jacobi algebroids respectively.  
These structures are defined as skew-symmetric tensor fields on algebroids.  A Poisson structure on a Lie algebroid $A$ induces a Lie algebroid structure on its dual bundle $A^*$.  
Similarly, a Jacobi structure on a Jacobi algebroid induces a Jacobi algebroid structure on its dual bundle.  
The Poissonization of a Jacobi structure on a Jacobi algebroid $(A,\phi_0)$ yields a Poisson structure on a Lie algebroid $A\times \mathbb{R}$.  The Poissonization of a Jacobi structure is a generalization of the symplectization of a contact structure. \\
A Koszul-Vinberg manifold was introduced in \cite{BB2} (it is called a contravariant pseudo-Hessian manifold in \cite{BB2}) as a generalization of a (pseudo-)Hessian manifold. Their relation is similar to the relation between Poisson manifolds and symplectic manifolds.  A non-degenerate Koszul-Vinberg manifold is equivalent to a (pseudo-)Hessian manifold.  Koszul-Vinberg structures on manifolds also extend to the structures on left-symmetric algebroids.  
A Koszul-Vinberg structure on a left-symmetric algebroid is defined as a symmetric tensor field, and is regarded as a symmetric analogue of a Poisson structure on a Lie algebroid. 
Liu, Sheng and Bai proved in \cite{LSB1} that a Koszul-Vinberg structure on a left-symmetric algebroid $A$ induces a left-symmetric algebroid structure on its dual bundle $A^*$.  Besides this, there are other studies on Koszul-Vinberg structures in terms of analogy of Poisson geometry such as \cite{ABB1}, \cite{ABB2}, \cite{BB2} and \cite{WLS}. \\ 
In this paper, we define a Jacobi-left-symmetric algebroid and a Jacobi-Koszul-Vinberg structure on a Jacobi-left-symmetric algebroid as a symmetric analogue of a Jacobi structure on a Jacobi algebroid.  We show that a Jacobi-Koszul-Vinberg structure on a Jacobi-left-symmetric algebroid $A$ induces a Jacobi-left-symmetric algebroid structure on its dual bundle $A^*$.  We also define the Koszul-Vinbergization of a Jacobi-Koszul-Vinberg structure as an analogy of the Poissonization of a Jacobi structure.  Furthermore, we give the definition of a Jacobi-Koszul-Vinberg manifold, which is a symmetric analogue of a Jacobi manifold.  Matsuzoe introduced in \cite{Mat} a semi-Weyl manifold as a generalization of both a statistical manifold and a Weyl manifold.  We prove that a Jacobi-Koszul-Vinberg manifold with non-degenerate symmetric $(2,0)$-tensor field $h$ is a semi-Weyl manifold.  Moreover, we also show that such a Jacobi-Koszul-Vinberg manifold is a locally conformally Hessian manifold defined by Osipov in \cite{Osi}. \\
This paper is organized as follows.  In Section 2, we recall the definitions and properties of several notions such as Lie algebroids, Poisson structures, left-symmetric algebroids and Koszul-Vinberg structures.  We explain similarities between Poisson structures on Lie algebroids and Koszul-Vinberg structures on left-symmetric algebroids.  We also review the definitions of Jacobi algebroids and Jacobi structures, which are generalizations of Lie algebroids and Poisson structures respectively.  We describe the Poissonization of a Jacobi structure at the end of Section 2.  
In Section 3, we introduce a Jacobi-left-symmetric algebroid as a generalization of a left-symmetric algebroid.  We define a Jacobi-Koszul-Vinberg structure on a Jacobi-left-symmetric algebroid as a generalization of a Koszul-Vinberg structure on a left-symmetric algebroid.  We show that a Jacobi-Koszul-Vinberg structure on a Jacobi-left-symmetric algebroid $A$ induces a Jacobi-left-symmetric algebroid structure on its dual bundle $A^*$.  We also define the Koszul-Vinbergization of a Jacobi-Koszul-Vinberg structure as an analogy of the Poissonization of a Jacobi structure.  At the end of Section 3, we give the definition of a Jacobi-Koszul-Vinberg manifold, which is considered as a symmetric analogue of a Jacobi manifold.  A Jacobi-Koszul-Vinberg manifold with non-degenerate symmetric $(2,0)$-tensor field $h$ is a semi-Weyl manifold.  Furthermore, we show that such a Jacobi-Koszul-Vinberg manifold is also a locally conformally Hessian manifold. 

\section{Preliminaries}
In this section, we recall the definitions and properties of Lie, left-symmetric and Jacobi algebroids, as well as Poisson, Koszul-Vinberg and Jacobi structures on algebroids. 

\subsection{Lie algebroids and Poisson structures}
A {\it Lie algebroid} over a manifold $M$ is a vector bundle
$A\rightarrow M$ equipped with a Lie bracket 
$[\cdot,\cdot]_{A}:\Gamma (A)\times \Gamma(A)\rightarrow \Gamma (A)$ 
and a bundle map $\rho_{A}:A\rightarrow TM$ over $M$, called the {\it anchor}, satisfying the following condition: for any $X,Y$ in $\Gamma (A)$ and $f$ in $C^{\infty }(M)$,
 \begin{align*}
[X,fY]_{A}=f[X,Y]_{A}+(\rho_{A}(X)f)Y,
 \end{align*}
where we denote the map $\Gamma(A)\rightarrow \Gamma(TM)=\mathfrak{X}(M)$ induced by the anchor, the same symbol $\rho_{A}$. 
For any Lie algebroid $(A,[\cdot,\cdot]_A,\rho_A)$ over $M$, it follows that for any $X$ and $Y$ in $\Gamma(A)$,
 \begin{align*}
\rho_A([X,Y]_{A})=[\rho_A(X),\rho_A(Y)],
 \end{align*}
where the bracket on the right hand side is the usual Lie bracket on $\mathfrak{X}(M)$.

\begin{example}

For any manifold $M$, the tangent bundle $(TM,[\cdot,\cdot],\mbox{id}_{TM})$ is a Lie algebroid over $M$, where $[\cdot,\cdot]$ is the usual Lie bracket on the vector fields $\mathfrak{X}(M)=\Gamma(TM)$. 
%
\end{example}

\begin{example}\label{A oplus R}
Let $A$ be a Lie algebroid over a manifold $M$ and set $A\oplus \mathbb{R}:=A\oplus(M\times \mathbb{R})$. 
Then $(A\oplus \mathbb{R},[\cdot,\cdot]_{A\oplus\mathbb{R}},\rho_{A}\circ\mbox{pr}_1)$ is also a Lie algebroid over $M$, where the bracket $[\cdot,\cdot]_{A\oplus\mathbb{R}}$ is defined by
\begin{align*}\label{A oplus R no kakko}
[(X,f),(Y,g)]_{A\oplus\mathbb{R}}&:=([X,Y]_{A},\rho_{A}(X)g-\rho_{A}(Y)f)
\end{align*}
for any $(X,f),(Y,g)$ in $\Gamma(A\oplus \mathbb{R})\cong \Gamma(A)\times C^\infty(M)$ and the map $\mbox{pr}_1:A\oplus \mathbb{R}\rightarrow A$ is the canonical projection to the first factor. 
\end{example}

Let $(A,[\cdot,\cdot]_A,\rho_A)$ be a Lie algebroid over $M$. The {\it Schouten bracket} on $\Gamma (\Lambda ^*A)$ is defined similarly to the Schouten bracket $[\cdot ,\cdot ]$ on the multivector fields $\mathfrak{X}^*(M)$ and is denoted by the same symbol $[\cdot,\cdot]_A$ as the bracket on $\Gamma (A)$. 
The {\it differential} of the Lie algebroid $A$ is an operator $d_A:\Gamma (\Lambda ^k A^*)\rightarrow \Gamma (\Lambda ^{k+1} A^*)$ defined by, for any $\omega $ in $\Gamma (\Lambda^k A^*)$ and $X_0,\dots ,X_k$ in $\Gamma(A)$,
\begin{align*}
(d_A \omega)(X_0, \dots ,X_k)&=\sum_{i=0}^k(-1)^i\rho_A (X_i)(\omega(X_0,\dots ,\hat{X_i},\dots ,X_k))\nonumber \\
&\quad+\sum _{i<j}^{}(-1)^{i+j}\omega ([X_i,X_j]_A , X_0,\dots ,\hat{X_i},\dots ,\hat{X_j},\dots ,X_k).
\end{align*}
The differential $d_A$ satisfies $d_A^2=0$. 
For any $X$ in $\Gamma (A)$, the {\it Lie derivative\index{Lie derivative}} $\mathcal{L}_X^A:\Gamma (\Lambda ^k A^*)\rightarrow \Gamma (\Lambda ^k A^*)$ is defined by the {\it Cartan formula} $\mathcal{L}_X^A:=d_A \iota _X +\iota _X d_A$ and $\mathcal{L}_X^A$ is extended on $\Gamma (\Lambda ^*A)$ in the same way as the usual Lie derivative $\mathcal{L}_X$. Then it follows that $\mathcal{L}_X^AD=[X,D]_A$ for any $D$ in $\Gamma (\Lambda^{*}A)$. We call a $d_{A}$-closed $2$-cosection $\omega$ in $\Gamma (\Lambda^2 A^*)$, i.e., $d_{A}\omega=0$, a {\it presymplectic structure} on $(A,[\cdot,\cdot]_A,\rho_A)$. A presymplectic structure $\omega$ is called a {\it symplectic structure} if $\omega$ is non-degenerate.


{\it A Poisson structure} on a Lie algebroid $A$ over a manifold $M$ is a $2$-section $\pi$ in $\Gamma(\Lambda^2 A)$ satisfying $[\pi,\pi]_A=0$. For any $2$-section $\pi $ in $\Gamma (\Lambda^2A)$, we define a skew-symmetric bilinear bracket $[\cdot,\cdot]_{\pi} $ on $\Gamma (A^*)$ by for any $\xi,\eta$ in $\Gamma (A^*)$,
\begin{align*}
[\xi,\eta]_{\pi}:=\mathcal{L}_{\pi^\sharp \xi}^{A}\eta -\mathcal{L}_{\pi^\sharp \eta}^{A}\xi -d_{A}\langle \pi^\sharp \xi,\eta \rangle,
\end{align*}
where a bundle map $\pi^\sharp:A^{*}\rightarrow A$ over $M$ is defined by $\langle \pi^\sharp \xi,\eta\rangle:=\pi(\xi,\eta)$. 
Moreover it follows that
 \begin{align*}
\frac{1}{2}[\pi ,\pi ]_{A}(\xi,\eta,\cdot)=[\pi^\sharp \xi,\pi^\sharp \eta ]_{A}-\pi^\sharp[\xi,\eta]_{\pi}.
\end{align*}
A triple $A^*_\pi:=(A^*,[\cdot,\cdot]_\pi,\rho_\pi)$, where $\rho_\pi:=\rho_A\circ\pi^\sharp$, is a Lie algebroid if and only if $\pi$ is Poisson on $(A, [\cdot,\cdot]_A,\rho_A)$. 

It is well known that there exists a one-to-one correspondence between symplectic structures and non-degenerate Poisson structures on a Lie algebroid $(A, [\cdot,\cdot]_A,\rho_A)$. In fact, for a non-degenerate Poisson structure $\pi$, a $2$-cosection $\omega_{\pi}$ characterized by $\omega_{\pi}^{\flat}=-(\pi^{\sharp})^{-1}$ is symplectic, where for any $2$-cosection $\Omega$, a bundle map $\Omega^\flat:A\rightarrow A^{*}$ over $M$ is defined by $\langle \Omega^\flat X,Y\rangle:=\Omega(X,Y)$ for any $X$ and $Y$ in $\Gamma(A)$.

An {\it affine connection} on a Lie algebroid $(A, [\cdot,\cdot]_{A}, \rho_{A})$ over $M$ is an $\mathbb{R}$-bilinear map 
$\nabla: \Gamma(A)\times \Gamma(A)\rightarrow \Gamma(A)$ satisfying 
for any $f \in C^{\infty}(M)$ and $X ,Y \in \Gamma(A)$, 
\begin{align*}
\nabla_{fX}Y&=f\nabla_X Y, \\
\nabla_X (fY)&=f\nabla_X Y+(\rho_A(X)f)Y.
\end{align*}
An affine connection on the standard Lie algebroid $(TM,[\cdot,\cdot],\mbox{id}_{TM})$ is just an affine connection on $M$.  An affine connection $\nabla$ on a Lie algebroid $(A, [\cdot,\cdot]_{A}, \rho_{A})$ over $M$
 is {\it torsion-free} if 
\begin{equation*}
\nabla_X Y -\nabla_Y X=[X,Y]_A,
\end{equation*}
and is {\it flat} if 
\begin{equation*}
\nabla_X\nabla_Y Z -\nabla_Y\nabla_X Z=\nabla_{[X,Y]_A}Z,
\end{equation*}
for any $X,Y$ and $Z \in \Gamma(A)$.

\subsection{Left-symmetric algebroids and Koszul-Vinberg structures}

A {\it statistical manifold} $(M, g, \nabla)$ is a (pseudo-)Riemannian manifold $(M,g)$ equipped with a torsion-free affine connection $\nabla$ on $M$
such that $\nabla g$ is symmetric, i.e., 
for any $X,Y,Z$ in $\mathfrak{X}(M)$,
\begin{equation}\label{Cod}
(\nabla_X g)(Y,Z)=(\nabla_Y g)(X,Z).
\end{equation}
The equation (\ref{Cod}) is called the {\it Codazzi equation}.

Let $(M, g)$ be a (pseudo-)Riemannian manifold and $\nabla$ an affine connection on $M$.  
The {\it dual connection} $\nabla^*$ of $\nabla$ with respect to $g$ is defined by, for $X,Y,Z$ in $\mathfrak{X}(M)$,
\begin{equation*}
X (g(Y,Z))=g(\nabla_X Y,Z)+g(Y,\nabla^*_X Z). 
\end{equation*}
It is known that for a torsion-free connection $\nabla$, 
the dual connection $\nabla^*$ is torsion-free if and only if $\nabla g$ is symmetric.  Hence a statistical manifold is defined also as the quadruple $(M, g, \nabla, \nabla^*)$ such that both $\nabla$ and $\nabla^*$ are torsion-free. 

For a statistical manifold $(M, g, \nabla)$, 
the connection $\nabla$ is flat if and only if the dual connection $\nabla^*$ is flat.  
Therefore, a statistical manifold $(M, g, \nabla)$ with the flat connection $\nabla$ is called a {\it dually flat manifold}.  
A dually flat manifold is also called a {\it Hessian manifold} since it possesses a local coordinate system in which the metric $g$ is described as the Hessian matrix of some local function with respect to the local coordinates. 

An {\it affine manifold} $(M, \nabla)$ is a manifold $M$ with a torsion-free flat connection $\nabla$ on $M$.  Hence a Hessian manifold $(M, g, \nabla)$ is an affine manifold (M,$\nabla$) with a (pseudo-)Riemannian metric $g$ satisfying the Codazzi equation (\ref{Cod}).  For more comprehensive study on Hessian manifolds, see \cite{Shi}.

\begin{defn}(\cite[Definition 6.6]{BB2})
A {\it Koszul-Vinberg structure} on an affine manifold $(M, \nabla)$ is a symmetric $(2,0)$-tensor field $h$ in $\Gamma (S^2TM)$ satisfying, for any $\alpha,\beta,\gamma$ in $\Omega^1(M)$,
\begin{equation}\label{hCod}
(\nabla_{h^{\sharp}\alpha} h)(\beta,\gamma)=(\nabla_{h^{\sharp}\beta} h)(\alpha,\gamma). 
\end{equation}
$(M, \nabla,h)$ is called a {\it Koszul-Vinberg manifold}.  The equation (\ref{hCod}) is also called the Codazzi equation.
\end{defn}
\noindent Note that a Koszul-Vinberg structure is called a contravariant pseudo-Hessian structure in \cite{BB2}.  

There exists a one-to-one correspondence between the non-degenerate symmetric $(0,2)$-tensor fields on $M$ and the non-degenerate symmetric $(2,0)$-tensor fields on $M$.  
For a non-degenerate symmetric $(2,0)$-tensor $h$, let $g$ be the corresponding non-degenerate symmetric $(0,2)$-tensor, that is, $g^{\flat}= (h^{\sharp})^{-1}$.  
Then we can easily check that the non-degenerate symmetric $(0,2)$-tensor $h$ satisfies the Codazzi equation (\ref{hCod}) if and only if $\nabla g$ is symmetric.  
Therefore, a non-degenerate Koszul-Vinberg manifold is equivalent to a (pseudo-)Hessian manifold.

An algebra $(A, \cdot)$ is called a {\it left-symmetric algebra} if it satisfies for any $x,y,z$ in $A$,
\begin{equation*}
(x,y,z)=(y,x,z),  
\end{equation*}
where $(x,y,z):=(x\cdot y)\cdot z-x\cdot(y\cdot z)$.  
A left-symmetric algebra is also called a {\it Koszul-Vinberg algebra}.

\begin{example}
An associative algebra is a left-symmetric algebra.
\end{example}
\begin{example}
Let $(M, \nabla)$ be an affine manifold.  
Then $(\mathfrak{X}(M), \cdot)$ is a left-symmetric algebra, where 
$X\cdot Y :=\nabla_X Y$ for $X,Y$ in $\mathfrak{X}(M)$.
\end{example}
A left-symmetric algebra $(A,\cdot)$ is Lie-admissible, i.e., $(A, [\cdot,\cdot])$ is a Lie algebra, where
$[x,y]:=x\cdot y-y\cdot x$ for $x,y$ in $A$.  
Hence a left-symmetric algebra is also called a {\it pre-Lie algebra}.

\begin{defn}(\cite[Definition 2.1]{Boy1})
A {\it left-symmetric algebroid} over a manifold $M$ is a vector bundle
$A\rightarrow M$ equipped with a left-symmetric algebra structure $\cdot_A$ on $\Gamma (A)$ and 
a bundle map $\rho_{A}:A\rightarrow TM$ over $M$, called the anchor, satisfying the following conditions$:$ for any $X,Y$ in $\Gamma (A)$ and $f$ in $C^{\infty }(M)$,
 \begin{align}
\label{left-Leib}
X\cdot_A (fY)&=f(X\cdot_A Y)+(\rho_{A}(X)f)Y, \\
\label{right-lin}
(fX)\cdot_A Y&=f(X\cdot_A Y). 
\end{align}
 \end{defn}
\noindent Note that a left-symmetric algebroid is called a Koszul-Vinberg algebroid in \cite{Boy1}. 

For a left-symmetric algebroid $(A,\cdot_A, \rho_A)$ over a manifold $M$, $(A, [\cdot,\cdot]_A, \rho_A)$ is a Lie algebroid over $M$, where $[X,Y]_A :=X\cdot_A Y-Y\cdot_A X$ for $X,Y$ in $\Gamma (A)$.  
$(A, [\cdot,\cdot]_A, \rho_A)$ is called the {\it sub-adjacent Lie algebroid} of the left-symmetric algebroid $(A,\cdot_A, \rho_A)$.

\begin{example}
Let $(A, [\cdot,\cdot]_A, \rho_A)$ be a Lie algebroid over $M$ and $\nabla$ a torsion-free flat connection on $(A, [\cdot,\cdot]_A, \rho_A)$.  Then $(A,\cdot_A, \rho_A)$ is a left-symmetric algebroid over $M$, where $X\cdot_A Y :=\nabla_X Y$ for $X,Y$ in $\Gamma(A)$.  The sub-adjacent Lie algebroid of $(A,\cdot_A, \rho_A)$ is the given Lie algebroid $(A, [\cdot,\cdot]_A, \rho_A)$.  

In particular, for an affine manifold $(M, \nabla)$, $TM_{\nabla}:=(TM,\nabla,\mbox{id}_{TM})$ is a left-symmetric algebroid over $M$.  The sub-adjacent Lie algebroid of $TM_{\nabla}$ is the standard Lie algebroid $(TM,[\cdot,\cdot],\mbox{id}_{TM})$.
\end{example}

\begin{example}\label{barnabla}
Let $(A, [\cdot,\cdot]_A, \rho_A)$ be a Lie algebroid over $M$ and $\nabla$ a torsion-free flat connection on $(A, [\cdot,\cdot]_A, \rho_A)$.  Define $\bar{\nabla}:\Gamma(A\oplus \mathbb{R})\times
\Gamma(A\oplus \mathbb{R})\rightarrow \Gamma(A\oplus \mathbb{R})$ by
\begin{equation*}
\bar{\nabla}_{(X,f)}(Y,g):=(\nabla_X Y, \rho_A(X)g)
\end{equation*}
for $(X,f),(Y,g)$ in $\Gamma(A\oplus \mathbb{R})\cong \Gamma(A)\times C^\infty(M)$.  Then $\bar{\nabla}$ is a torsion-free flat connection on the Lie algebroid $(A\oplus \mathbb{R},[\cdot,\cdot]_{A\oplus\mathbb{R}},\rho_{A}\circ\mbox{pr}_1)$ in Example \ref{A oplus R}.  Hence $(A\oplus \mathbb{R},\bar{\nabla},\rho_{A}\circ\mbox{pr}_1)$ is a left-symmetric algebroid over $M$ and its sub-adjacent Lie algebroid is $(A\oplus \mathbb{R},[\cdot,\cdot]_{A\oplus\mathbb{R}},\rho_{A}\circ\mbox{pr}_1)$.

In particular, for an affine manifold $(M, \nabla)$, $(TM\oplus \mathbb{R},\bar{\nabla},\mbox{pr}_1)$ is a left-symmetric algebroid over $M$.
\end{example}

Let $(A,\cdot_A, \rho_A)$ be a left-symmetric algebroid over a manifold $M$.  
For a symmetric $(2,0)$-tensor $h$ in $\Gamma (S^2A)$, we define $\llbracket h,h \rrbracket_A$ in $\Gamma (\Lambda^2A\otimes A)$ by, for $\alpha,\beta,\gamma$ in $\Gamma (A^*)$,
 \begin{align*}
\llbracket h,h \rrbracket_A (\alpha,\beta, \gamma)=
\rho_{A}&(h^{\sharp}\alpha) h(\beta,\gamma)-\rho_{A}(h^{\sharp}\beta) h(\alpha,\gamma) 
+\langle \alpha, h^{\sharp}\beta \cdot_A h^{\sharp}\gamma \rangle \\
&-\langle \beta, h^{\sharp}\alpha \cdot_A h^{\sharp}\gamma \rangle 
-\langle \gamma, [h^{\sharp}\alpha, h^{\sharp}\beta ]_A \rangle.
\end{align*}
The bracket $\llbracket \cdot,\cdot \rrbracket_A$ is an analogy of 
the Schouten bracket $[\cdot,\cdot]_{A}$ on $\Gamma (\Lambda ^*A)$ for a Lie algebroid $(A, [\cdot,\cdot]_A, \rho_A)$.  
Set $C^{k+1}(A):= \Gamma (\Lambda ^k A^* \otimes A^*)$ for $k\geq 0$.  
The coboundary operator $\delta_A:C^{k}(A)\rightarrow C^{k+1}(A)$ is defined by, for $\omega$ in $C^{k}(A)$ and $X_0,\dots ,X_k$ in $\Gamma(A)$,
\begin{small}
 \begin{align*}
(\delta_A \omega)(X_0, \dots ,X_k)=&\sum_{i=0}^{k-1}(-1)^i\rho_A (X_i)(\omega(X_0,\dots ,\hat{X_i},\dots ,X_k)) \\
&-\sum_{i=0}^{k-1}(-1)^i \omega(X_0,\dots ,\hat{X_i},\dots ,X_{k-1}, X_i \cdot_A X_k) \\
&+\sum _{0\leq i<j\leq k-1} (-1)^{i+j}\omega ([X_i,X_j]_A , X_0,\dots ,\hat{X_i},\dots ,\hat{X_j},\dots ,X_k).
\end{align*}
\end{small}

\begin{defn}(\cite{LSB1})
A Koszul-Vinberg structure on a left-symmetric algebroid $(A,\cdot_A,$ $ \rho_A)$ over a manifold $M$ is a symmetric $(2,0)$-tensor field $h $ in $\Gamma (S^{2}A)$ such that
\begin{equation*}
\llbracket h,h \rrbracket_A =0.
\end{equation*} 
\end{defn}  

For the left-symmetric algebroid $TM_{\nabla}$ and $h$ in $\Gamma (S^2TM)$, 
the condition $\llbracket h,h \rrbracket_{TM_{\nabla}} =0$ is equivalent to the Codazzi equation (\ref{hCod}).  
This means that a Koszul-Vinberg structure on an affine manifold $(M, \nabla)$ is a Koszul-Vinberg structure on the left-symmetric algebroid $TM_{\nabla}$.  
On the other hand, for $g \in \Gamma (S^2 T^*M) \subset C^2(TM)=\Gamma (T^*M\otimes T^*M)$, 
the condition $\delta_{TM_{\nabla}} \: g =0$ is equivalent to that $\nabla g$ is symmetric.  
Recall that for non-degenerate $h$ and the corresponding non-degenerate $g$, that is, $g^{\flat}= (h^{\sharp})^{-1}$, $h$ satisfies the Codazzi equation (\ref{hCod}) if and only if $\nabla g$ is symmetric.  
Hence, if $h$ in $\Gamma (S^2TM)$ is non-degenerate, $\llbracket h,h \rrbracket_{TM_{\nabla}} =0$ is equivalent to $\delta_{TM_{\nabla}} \: g =0$, where $g$ in $\Gamma (S^2T^*M)$ is the corresponding symmetric $(0,2)$-tensor field to $h$.  
This equivalence on the left-symmetric algebroid $TM_{\nabla}$ extends to a general left-symmetric algebroid in the following proposition.

\begin{prop}{\rm (\cite[Proposition 4.7]{LSB1})}
Let $(A,\cdot_A, \rho_A)$ be a left-symmetric algebroid over a manifold $M$.  
If $h$ in $\Gamma (S^2A)$ is non-degenerate, then 
$\llbracket h,h \rrbracket_A =0$ is equivalent to $\delta_A \: g =0$, where $g$ in $\Gamma (S^2A^*)$ is the corresponding section to $h$.
\end{prop}
This equivalence on left-symmetric algebroids is regarded as a symmetric analogue of the correspondence between non-degenerate Poisson structures and symplectic structures on Lie algebroids mentioned at the end of Subsection 2.1. 

As previously stated in Subsection 2.1, a Poisson structure $\pi$ on a Lie algebroid $(A, [\cdot,\cdot]_A,\rho_A)$ induces a Lie algebroid structure on the dual bundle $A^*$, denoted by $A^*_\pi$.  A similar property holds for a Koszul-Vinberg structure on a left-symmetric algebroid in the following way.
\begin{prop}{\rm (\cite[Theorem 4.10]{LSB1})}\label{Thm4.10 in [7]}
Let $(A,\cdot_A, \rho_A)$ be a left-symmetric algebroid over a manifold $M$ and 
$h$ in $\Gamma (S^2A)$ a Koszul-Vinberg structure on $(A,\cdot_A, \rho_A)$.  
Then $(A^*, \cdot^{h^{\sharp}}, \rho_{A} \circ h^{\sharp})$ is a left-symmetric algebroid over $M$, 
where $\cdot^{h^{\sharp}}$ is defined by, for $\alpha, \beta$ in $\Gamma(A^*)$,
\begin{equation}\label{K-V product}
\alpha \cdot^{h^{\sharp}} \beta :=\mathcal{L}^A_{h^\sharp \alpha}\beta -R_{h^\sharp \beta}\alpha -d_A(h (\alpha,\beta)).
\end{equation} 
Here $\mathcal{L}^A_X$ and $d_A$ are the Lie derivative and the differential respectively of the sub-adjacent Lie algebroid of the left-symmetric algebroid $(A,\cdot_A, \rho_A)$.  
The operator $R_X:\Gamma(A^*)\rightarrow \Gamma(A^*)$ is defined by, for $\alpha$ in $\Gamma(A^*)$ and $X,Y$ in $\Gamma(A)$,
\begin{equation*}
\langle R_X \alpha, Y\rangle = -\langle \alpha, Y\cdot_A X\rangle. 
\end{equation*}
\end{prop}
\noindent
The other operator $L_X:\Gamma(A^*)\rightarrow \Gamma(A^*)$ is defined by, for any $\alpha$ in $\Gamma(A^*)$ and $X,Y$ in $\Gamma(A)$,
\begin{align*}
\langle L_X \alpha, Y\rangle = \rho_A(X)\langle \alpha, Y\rangle-\langle \alpha, X\cdot_A Y\rangle.
\end{align*}
For any $h$ in $\Gamma(S^2A^*)$, the product $\cdot^{h^\sharp}$ defined by (\ref{K-V product}) satisfies the following equations:
\begin{align}
&\llbracket h,h\rrbracket_A(\alpha,\,\cdot\,,\beta)=h^\sharp(\alpha\cdot^{h^\sharp}\beta)-h^\sharp\alpha\cdot_Ah^\sharp\beta,\label{h,h A property}\\
&\langle(\alpha\cdot^{h^\sharp}\beta)\cdot^{h^\sharp}\gamma-\alpha\cdot^{h^\sharp}(\beta\cdot^{h^\sharp}\gamma)-((\beta\cdot^{h^\sharp}\alpha)\cdot^{h^\sharp}\gamma-\beta\cdot^{h^\sharp}(\alpha\cdot^{h^\sharp}\gamma)),X\rangle \label{left-symmetry of cdot h sharp}\\
&\ =\langle\mathcal{L}_{\llbracket h,h\rrbracket_A(\alpha,\,\cdot\,,\beta)-\llbracket h,h\rrbracket_A(\beta,\,\cdot\,,\alpha)}^A\gamma,X\rangle 
+\llbracket h,h\rrbracket_A(\beta,L_X\alpha,\gamma)-\llbracket h,h\rrbracket_A(\alpha,L_X\beta,\gamma) \nonumber 
\end{align}
for any $\alpha,\beta,\gamma$ in $\Gamma(A^*)$ and $X$ in $\Gamma(A)$.

\subsection{Jacobi algebroids and Jacobi structures}\label{Jacobi algebroids and Jacobi structures}

A pair $(A,\phi_0)$ is a {\it Jacobi algebroid} over a manifold $M$ if $A=(A,[\cdot ,\cdot ]_A,\rho_A)$ is a Lie algebroid over $M$ and $\phi_0$ in $\Gamma(A^*)$ is $d_A$-closed, that is, $d_A\phi_{0}=0$. 

\begin{example}\label{trivial example}
For any Lie algebroid $A$ over $M$, we set $\phi_0:=0$. Then $(A,\phi_0)$ is a Jacobi algebroid. We call $\phi_0$ the {\it trivial Jacobi algebroid structure} on $A$. Therefore any Lie algebroid is a Jacobi algebroid.
\end{example}

\begin{example}\label{01cocy}
For a Lie algebroid $A\oplus \mathbb{R}$ in Example \ref{A oplus R}, we set $\phi_0:=(0,1)$ in $\Gamma(A^{*}\oplus \mathbb{R})=\Gamma (A^{*})\times C^{\infty}(M)$. Then $(A\oplus \mathbb{R},\phi_0)$ is a Jacobi algebroid.
\end{example}

For a Jacobi algebroid $(A,\phi_{0})$, the {\it $\phi_0$-Schouten bracket} $[\cdot,\cdot]_{A,\phi_0}$ on $\Gamma (\Lambda^*A)$ is given by
\begin{align*}
[D_1,D_2]_{A,\phi_0}:=[D_1,D_2]_A+(a_1-1)&D_1\wedge \iota _{\phi_0}D_2 \nonumber\\
                                                     &-(-1)^{a_1+1}(a_2-1)\iota_{\phi_0}D_1\wedge D_2
\end{align*}
for any $D_i$ in $\Gamma(\Lambda ^{a_i}A)$, where $[\cdot,\cdot]_A$ is the Schouten bracket of the Lie algebroid $A$. The {\it $\phi_0$-differential} $d_{A,\phi_0}$ and the {\it $\phi_0$-Lie derivative} $\mathcal{L}_X^{A,\phi_0}$ are defined by
\begin{equation}\label{phi_0-differential and Lie derivative}
d_{A,\phi_0}\omega :=d_A\omega +\phi_0\wedge \omega,\quad \mathcal{L}_X^{A,\phi_0}:=\iota _X\circ d_{A,\phi_0}+d_{A,\phi_0}\circ \iota_X
\end{equation}
for any $\omega$ in $\Gamma (\Lambda ^*A^*)$ and $X$ in $\Gamma (A)$.

A {\it Jacobi structure} on a Jacobi algebroid $(A,\phi_{0})$ is a $2$-section $\pi $ in $\Gamma (\Lambda^{2}A)$ satisfying the condition
\begin{equation*}\label{Jacobi def equation}
[\pi,\pi]_{A,\phi_{0}}=0.
\end{equation*}
For a Lie algebroid $A$, a Poisson structure $\pi$ on $A$ is a Jacobi structure on the trivial Jacobi algebroid $(A,0)$.

For any $2$-section $\pi $ on $(A,\phi_0)$, we define a skew-symmetric bilinear bracket $[\cdot,\cdot]_{\pi,\phi_{0}} $ on $\Gamma (A^*)$ by, for any $\xi,\eta$ in $\Gamma (A^*)$,
\begin{align*}
[\xi,\eta]_{\pi,\phi_{0}}:=\mathcal{L}_{\pi^\sharp \xi}^{A,\phi_0}\eta -\mathcal{L}_{\pi^\sharp \eta}^{A,\phi_0}\xi -d_{A,\phi_0}\langle \pi^\sharp \xi,\eta \rangle.
\end{align*}
Moreover it follows that
 \begin{align*}
\frac{1}{2}[\pi ,\pi ]_{A,\phi_0}(\xi,\eta,\cdot)=[\pi^\sharp \xi,\pi^\sharp \eta ]_{A}-\pi^\sharp[\xi,\eta]_{\pi,\phi_{0}}.
\end{align*}
Then a triple $A^*_{\pi,\phi_0} := (A^*, [\cdot,\cdot]_{\pi,\phi_0}, \rho_\pi)$ is a Lie algebroid over $M$ if and only if $\pi$ is Jacobi. Furthermore, in the case that $\pi$ is Jacobi, a pair $(A^*_{\pi,\phi_0}, X_0)$ is a Jacobi algebroid over $M$, where $X_0:=-\pi^{\sharp}\phi_0$ in $\Gamma(A)$. 


\begin{example}\label{Jacobi manifold}
Let $A$ be a Lie algebroid over $M$, $\Lambda $ a $2$-section on $A$ and $E$ a section on $A$ satisfying
\begin{align*}
[\Lambda ,\Lambda ]_{A}=2E\wedge \Lambda ,\quad [E,\Lambda ]_{A}=0.
\end{align*}
Then a pair $(\Lambda ,E)$ in $\Gamma (\Lambda ^{2}A)\oplus \Gamma (A)\cong \Gamma (\Lambda ^{2}(A\oplus \mathbb{R}))$ is a Jacobi structure on a Jacobi algebroid $(A\oplus \mathbb{R},(0,1))$, i.e., it satisfies $[(\Lambda ,E),(\Lambda ,E)]_{A\oplus \mathbb{R},(0,1)}=0$. When $(\Lambda ,E)$ is a Jacobi structure on $(TM\oplus \mathbb{R},(0,1))$, we call it a {\it Jacobi structure on $M$} and a triple $(M,\Lambda,E)$ a {\it Jacobi manifold}. If $\pi $ is a Poisson structure on $A$, then $(\pi,0)$ is a Jacobi structure on $(A\oplus \mathbb{R},(0,1))$.
\end{example}

%
Let $(A,\phi_0)$ be a Jacobi algebroid over $M$. We set $\tilde{A}:=A\times \mathbb{R}$. Then $\tilde{A}$ is a vector bundle over $M\times \mathbb{R}$. The sections $\Gamma (\tilde{A})$ can be identified with the set of time-dependent sections of $A$. Here a time-dependent section on $A$ means a section on $A$ with a parameter $t$, where $t$ is a coordinate of $\mathbb{R}$. Under this identification, we can define two Lie algebroid structures $([\cdot,\cdot\hat{]}_A^{\phi_0},\hat{\rho}_A^{\phi_0})$ and $([\cdot,\cdot\bar{]}_A^{\phi_0},\bar{\rho}_A^{\phi_0})$ on $\tilde{A}$, where for any $\tilde{X}$ and $\tilde{Y}$ in $\Gamma(\tilde{A})$,
\begin{align}
 [\tilde{X},\tilde{Y}\hat{]}_A^{\phi_0}&:=e^{-t}\left([\tilde{X},\tilde{Y}]_A+\langle\phi_0,\tilde{X}\rangle\left(\frac{\partial \tilde{Y}}{\partial t}-\tilde{Y}\right)-\langle\phi_0,\tilde{Y}\rangle\left(\frac{\partial \tilde{X}}{\partial t}-\tilde{X}\right)\right), \label{hat kakko}\\ 
 \hat{\rho}_A^{\phi_0}(\tilde{X})&:=e^{-t}\left(\rho_A(\tilde{X})+\langle \phi_0,\tilde{X}\rangle \frac{\partial}{\partial t}\right), \label{hat anchor}\\
 [\tilde{X},\tilde{Y}\bar{]}_A^{\phi_0}&:=[\tilde{X},\tilde{Y}]_A+\langle\phi_0,\tilde{X}\rangle\frac{\partial \tilde{Y}}{\partial t}-\langle\phi_0,\tilde{Y}\rangle\frac{\partial \tilde{X}}{\partial t}, \label{bar kakko}\\ 
 \bar{\rho}_A^{\phi_0}(\tilde{X})&:=\rho_A(\tilde{X})+\langle \phi_0,\tilde{X}\rangle \frac{\partial}{\partial t}. \label{bar anchor}
\end{align}
Conversely, for a Lie algebroid $A$ over $M$ and a section $\phi_0$ on $A$, if the triple $(\tilde{A},[\cdot,\cdot\hat{]}_A^{\phi_0},\hat{\rho}_A^{\phi_0})$ (resp. $(\tilde{A},[\cdot,\cdot\bar{]}_A^{\phi_0},\bar{\rho}_A^{\phi_0})$) defined by (\ref{hat kakko}) and (\ref{hat anchor}) (resp. (\ref{bar kakko}) and (\ref{bar anchor})) is a Lie algebroid over $M\times \mathbb{R}$, then $(A,\phi_0)$ is a Jacobi algebroid over $M$, i.e., $d_{A}\phi_0=0$. A vector bundle $\tilde{A}$ equipped with the Lie algebroid structure $([\cdot,\cdot\hat{]}_A^{\phi_0},\hat{\rho}_A^{\phi_0})$ (resp. $([\cdot,\cdot\bar{]}_A^{\phi_0},\bar{\rho}_A^{\phi_0})$) is denoted by $\tilde{A}_{{\phi}_{0}}^{\wedge}$ (resp. $\tilde{A}_{{\phi}_{0}}^{-}$). 

Let $(A,\phi_0)$ be a Jacobi algebroid over $M$, $\pi$ a 2-section on $A$ and set $\tilde{\pi}:=e^{-t}\pi$ in $\Gamma (\Lambda ^2\tilde{A})$. Then the following holds:
\begin{align*}
[\tilde{\pi},\tilde{\pi}\bar{]}_A^{\phi_0}=e^{-2t}[\pi,\pi]_{A,\phi_0}.
\end{align*}
Therefore a 2-section $\pi$ on $A$ is a Jacobi structure on a Jacobi algebroid $(A,\phi_0)$ over $M$ if and only if $\tilde{\pi}$ in $\Gamma (\Lambda ^2\tilde{A})$ is a Poisson structure on a Lie algebroid $\tilde{A}_{{\phi}_{0}}^{-}$ over $M\times \mathbb{R}$. The Poisson structure $\tilde{\pi}$ on $\tilde{A}_{{\phi}_{0}}^{-}$ is called {\it the Poissonization} of $\pi$. 

In the case of $(A,\phi_0)=(TM\oplus \mathbb{R}, (0,1))$, the Lie algebroid $\tilde{A}_{{\phi}_{0}}^{-}
$ is isomorphic to the standard Lie algebroid $T(M\times \mathbb{R})$ over $M\times \mathbb{R}$. Then the Poissonization $\widetilde{(\Lambda ,E)}$ of a Jacobi structure $(\Lambda ,E)$ on $(TM\oplus \mathbb{R},(0,1))$ corresponds to a Poisson structure $\Pi:=e^{-t}\left(\Lambda+\dfrac{\partial}{\partial t}\wedge E\right)$ on $T(M\times \mathbb{R})$. This is just the Poissonizaion of a Jacobi structure on $M$.

\section{Jacobi-left-symmetric algebroids and Jacobi-Koszul-Vinberg structures}
In this section, we define a Jacobi-left-symmetric algebroid and a Jacobi-Koszul-Vinberg structure on a Jacobi-left-symmetric algebroid.  We show that a Jacobi-Koszul-Vinberg structure on a Jacobi-left-symmetric algebroid $(A,\phi_0)$ induces a Jacobi-left-symmetric algebroid structure on its dual bundle $A^*$.  We also define the Koszul-Vinbergization of a Jacobi-Koszul-Vinberg structure as an analogy of the Poissonization of a Jacobi structure.  Furthermore, we give the definition of a Jacobi-Koszul-Vinberg manifold, which is a symmetric analogue of a Jacobi manifold. Finally, we show that a Jacobi-Koszul-Vinberg manifold with non-degenerate $h \in \Gamma(S^2 TM)$ is a locally conformally Hessian manifold defined in \cite{Osi}.   

\begin{defn}
A {\it Jacobi-left-symmetric algebroid} over $M$ is a pair $(A,\phi_0)$ of a left-symmetric algebroid $A=(A,\cdot _A,\rho_A)$ over $M$ and $\phi_0$ in $\Gamma(A^*)$ satisfying that $\delta_A\phi_{0}$ is symmetric. 
\end{defn}

\begin{prop}\label{cocyc}
Let $(A,\cdot_A, \rho_A)$ be a left-symmetric algebroid and $\phi_0$ an element in $\Gamma(A^*)$. Then $\delta_A \phi_0$ is symmetric if and only if $d_A\phi_0=0$, where $d_A$ is the differential of the sub-adjacent Lie algebroid of $(A,\cdot_A, \rho_A)$. In particular, if $\delta_A \phi_0 =0$, then $d_A\phi_0=0$. 
\end{prop}

\begin{proof}
Let $(A, [\cdot,\cdot]_A, \rho_A)$ be the sub-adjacent Lie algebroid of $(A,\cdot_A, \rho_A)$. For any $X$ and $Y$ in $\Gamma(A)$,
\begin{align*}
    (d_A\phi_0)(X,Y)&=\rho_A(X)\langle \phi_0,Y\rangle-\rho_A(Y)\langle \phi_0,X\rangle-\langle \phi_0,[X,Y]_A\rangle\\
                    &=\rho_A(X)\langle \phi_0,Y\rangle-\rho_A(Y)\langle\phi_0,X\rangle-\langle \phi_0,X\cdot_AY-Y\cdot_AX\rangle\\
                    &=(\rho_A(X)\langle \phi_0,Y\rangle-\langle \phi_0,X\cdot_AY\rangle)\\
                    &\qquad \qquad -(\rho_A(Y)\langle \phi_0,X\rangle-\langle \phi_0,Y\cdot_AX\rangle)\\
                    &=(\delta_A\phi_0)(X,Y)-(\delta_A\phi_0)(Y,X).
\end{align*}
Therefore the conclusion holds.
\end{proof}

\begin{corollary}\label{sub-adjacent Jacobi algebroid}
For a Jacobi-left-symmetric algebroid $((A,\cdot_A, \rho_A), \phi_0)$, a pair $((A,$ $ [\cdot,\cdot]_A, \rho_A), \phi_0)$ is a Jacobi algebroid, where $[X,Y]_A=X\cdot_A Y-Y\cdot_A X$ for any $X,Y$ in $\Gamma (A)$.
\end{corollary}
\noindent
We call $((A, [\cdot,\cdot]_A, \rho_A), \phi_0)$ in Corollary \ref{sub-adjacent Jacobi algebroid} the {\it sub-adjacent Jacobi algebroid} of $((A,\cdot_A, \rho_A),$ $ \phi_0)$.

Let $(A,\phi_0)$ be a Jacobi-left-symmetric algebroid over $M$.
For $h$ in $\Gamma (S^2A)$, define $\llbracket h,h \rrbracket_A^{\phi_0}$ in $\Gamma (\Lambda^2A\otimes A)$ by
\begin{align*}
\llbracket h,h \rrbracket_A^{\phi_0} (\alpha,\beta, \gamma):=
\: \llbracket h,h \rrbracket_A (\alpha,\beta, \gamma)+h(\phi_0,\alpha)h(\beta,\gamma) -h(\phi_0,\beta)h(\alpha,\gamma)
\end{align*}
for any $\alpha,\beta$  and $\gamma$ in $\Gamma (A^*)$. 
$\llbracket \cdot,\cdot \rrbracket_A^{\phi_0}$ is an analogy of the $\phi_0$-twisted Schouten bracket 
$[\cdot,\cdot]_{A}^{\phi_0}$ on $\Gamma (\Lambda ^*A)$ for a Jacobi algebroid $((A, [\cdot,\cdot]_A, \rho_A), \phi_0)$.

\begin{defn}
{\it A Jacobi-Koszul-Vinberg structure} $h$ on a Jacobi-left-symmetric algebroid $(A,\phi_0)$ over $M$ is a section of $S^{2}A$ such that
\begin{equation*}
\llbracket h,h \rrbracket_A ^{\phi_0}=0.
\end{equation*} 
\end{defn}  

\begin{theorem}
Let $((A,\cdot_A, \rho_A), \phi_0)$ be a Jacobi-left-symmetric algebroid over $M$ and $h$ an element in $\Gamma(S^2A)$. Then:
\begin{enumerate}
    \item[(1)] $A^*_{h,\phi_0}:= (A^*, \cdot^{h^{\sharp},\phi_0}, \rho_{A} \circ h^{\sharp})$ satisfies (\ref{left-Leib}) and (\ref{right-lin}), where $\cdot^{h^{\sharp},\phi_0}$ is defined by
\begin{equation*}
\alpha \cdot^{h^{\sharp},\phi_0} \beta :=\mathcal{L}^{A,\phi_0}_{h^\sharp \alpha}\beta -R_{h^\sharp \beta}\alpha 
-d_{A}^{\phi_0}(h (\alpha,\beta)). \quad (\alpha, \beta \in \Gamma(A^*))
\end{equation*} 
    \item[(2)] If $h$ is a Jacobi-Koszul-Vinberg structure on $((A,\cdot_A, \rho_A), \phi_0)$, then $A^*_{h,\phi_0}$ is a left-symmetric algebroid over $M$. Furthermore, $(A^*_{h,\phi_0}, -h^{\sharp}\phi_0)$ is a Jacobi-left-symmetric algebroid over $M$.
\end{enumerate}
\end{theorem}

\begin{proof}
(1) By (\ref{phi_0-differential and Lie derivative}), for any $\alpha$ and $\beta$ in $\Gamma(A^*)$, we have
\begin{align*}
\alpha \cdot^{h^{\sharp},\phi_0} \beta=\alpha \cdot^{h^{\sharp}} \beta+\langle \phi_0,h^\sharp\alpha\rangle\beta-h(\alpha,\beta)\phi_0.
\end{align*}
For any $f$ in $C^\infty(M)$,
\begin{align*}
\alpha \cdot^{h^{\sharp},\phi_0}(f\beta)&=\alpha \cdot^{h^{\sharp}} (f\beta)+\langle \phi_0,h^\sharp\alpha\rangle f\beta-h(\alpha,f\beta)\phi_0\\
&=f(\alpha \cdot^{h^{\sharp}} \beta)+(\rho_A(h^\sharp\alpha)f)\beta+f\langle \phi_0,h^\sharp\alpha\rangle\beta-fh(\alpha,\beta)\phi_0\\
&=f(\alpha \cdot^{h^{\sharp}} \beta+\langle \phi_0,h^\sharp\alpha\rangle\beta-h(\alpha,\beta)\phi_0)+(\rho_A(h^\sharp\alpha)f)\beta\\
&=f\alpha \cdot^{h^{\sharp},\phi_0}\beta+(\rho_A(h^\sharp\alpha)f)\beta,\\
(f\alpha) \cdot^{h^{\sharp},\phi_0}\beta&=(f\alpha) \cdot^{h^{\sharp}} \beta+\langle \phi_0,h^\sharp f\alpha\rangle \beta-h(f\alpha,\beta)\phi_0\\
&=f(\alpha \cdot^{h^{\sharp}} \beta)+f\langle \phi_0,h^\sharp \alpha\rangle \beta-fh(\alpha,\beta)\phi_0\\
&=f(\alpha \cdot^{h^{\sharp}} \beta+\langle \phi_0,h^\sharp \alpha\rangle \beta-h(\alpha,\beta)\phi_0)\\
&=f(\alpha \cdot^{h^{\sharp},\phi_0}\beta)
\end{align*}
hold. To show (2), we first prove that $\llbracket h,h\rrbracket_{A}^{\phi_0}(\alpha,\cdot,\beta)=h^\sharp(\alpha\cdot^{h^\sharp,\phi_0}\beta)-h^\sharp\alpha\cdot_Ah^\sharp\beta$. In fact, by using (\ref{h,h A property}),
\begin{align*}
\llbracket h,h \rrbracket_A^{\phi_0} (\alpha,\beta, \gamma)&=
\: \llbracket h,h \rrbracket_A (\alpha,\beta, \gamma)+h(\phi_0,\alpha)h(\beta,\gamma) -h(\phi_0,\beta)h(\alpha,\gamma)\\
&=\langle h^\sharp(\alpha\cdot^{h^\sharp}\gamma)-h^\sharp\alpha\cdot_Ah^\sharp\gamma,\beta\rangle+\langle \phi_0,h^\sharp\alpha\rangle\langle h^\sharp\gamma,\beta\rangle-h(\alpha,\gamma)\langle h^\sharp\phi_0,\beta\rangle\\
&=\langle h^\sharp(\alpha\cdot^{h^\sharp}\gamma+\langle \phi_0,h^\sharp\alpha\rangle\gamma-h(\alpha,\gamma)\phi_0)-h^\sharp\alpha\cdot_Ah^\sharp\gamma,\beta\rangle\\
&=\langle h^\sharp(\alpha\cdot^{h^\sharp,\phi_0}\gamma)-h^\sharp\alpha\cdot_Ah^\sharp\gamma,\beta\rangle.
\end{align*}
Then we calculate
\begin{align*}
(\alpha&\cdot^{h^\sharp,\phi_0}\beta)\cdot^{h^\sharp,\phi_0}\gamma\\
&=(\alpha \cdot^{h^{\sharp}} \beta+\langle \phi_0,h^\sharp\alpha\rangle\beta-h(\alpha,\beta)\phi_0)\cdot^{h^{\sharp}}\gamma\\
&\qquad +\langle\phi_0,h^\sharp(\alpha\cdot^{h^\sharp,\phi_0}\beta)\rangle\gamma-h(\alpha\cdot^{h^\sharp,\phi_0}\beta,\gamma)\phi_0\\
&=(\alpha\cdot^{h^\sharp}\beta)\cdot^{h^\sharp}\gamma+(\langle \phi_0,h^\sharp\alpha\rangle\beta)\cdot^{h^\sharp}\gamma-(h(\alpha,\beta)\phi_0)\cdot^{h^\sharp}\gamma\\
&\qquad +(\llbracket h,h \rrbracket_A^{\phi_0}(\alpha,\phi_0,\beta)+ 
\langle\phi_0,h^\sharp\alpha \cdot_A h^\sharp\beta\rangle)\gamma-\langle h^\sharp(\alpha \cdot^{h^{\sharp},\phi_0} \beta),\gamma\rangle\phi_0\\
&=(\alpha\cdot^{h^\sharp}\beta)\cdot^{h^\sharp}\gamma+h(\alpha,\phi_0)(\beta\cdot^{h^\sharp}\gamma)-h(\alpha,\beta)(\phi_0\cdot^{h^\sharp}\gamma)\\
&\qquad +\llbracket h,h \rrbracket_A^{\phi_0}(\alpha,\phi_0,\beta)\gamma+ 
\langle\phi_0,h^\sharp\alpha \cdot_A h^\sharp\beta\rangle\gamma\\
&\qquad -(\llbracket h,h \rrbracket_A^{\phi_0}(\alpha,\gamma,\beta)+\langle h^\sharp\alpha \cdot_A h^\sharp\beta,\gamma\rangle)\phi_0\\
&=(\alpha\cdot^{h^\sharp}\beta)\cdot^{h^\sharp}\gamma+h(\alpha,\phi_0)(\beta\cdot^{h^\sharp}\gamma)-h(\alpha,\beta)(\phi_0\cdot^{h^\sharp}\gamma)\\
&\qquad +\llbracket h,h \rrbracket_A^{\phi_0}(\alpha,\phi_0,\beta)\gamma+ 
\langle\phi_0,h^\sharp\alpha \cdot_A h^\sharp\beta\rangle\gamma\\
&\qquad -\llbracket h,h \rrbracket_A^{\phi_0}(\alpha,\gamma,\beta)\phi_0-\langle h^\sharp\alpha \cdot_A h^\sharp\beta,\gamma\rangle\phi_0,\\
(\beta&\cdot^{h^\sharp,\phi_0}\alpha)\cdot^{h^\sharp,\phi_0}\gamma\\
&=(\beta\cdot^{h^\sharp}\alpha)\cdot^{h^\sharp}\gamma+h(\beta,\phi_0)(\alpha\cdot^{h^\sharp}\gamma)-h(\beta,\alpha)(\phi_0\cdot^{h^\sharp}\gamma)\\
&\qquad +\llbracket h,h \rrbracket_A^{\phi_0}(\beta,\phi_0,\alpha)\gamma+ 
\langle\phi_0,h^\sharp\beta \cdot_A h^\sharp\alpha\rangle\gamma\\
&\qquad -\llbracket h,h \rrbracket_A^{\phi_0}(\beta,\gamma,\alpha)\phi_0-\langle h^\sharp\beta \cdot_A h^\sharp\alpha,\gamma\rangle\phi_0,\\
\alpha&\cdot^{h^\sharp,\phi_0}(\beta\cdot^{h^\sharp,\phi_0}\gamma)\\
&=\alpha\cdot^{h^\sharp}(\beta \cdot^{h^{\sharp}} \gamma+\langle \phi_0,h^\sharp\beta\rangle\gamma-h(\beta,\gamma)\phi_0)\\
&\qquad +\langle\phi_0,h^\sharp\alpha\rangle(\beta \cdot^{h^{\sharp}} \gamma+\langle \phi_0,h^\sharp\beta\rangle\gamma-h(\beta,\gamma)\phi_0)-h(\alpha,\beta\cdot^{h^\sharp,\phi_0}\gamma)\phi_0\\
&=\alpha\cdot^{h^\sharp}(\beta \cdot^{h^{\sharp}} \gamma)+(\rho_A(h^\sharp\alpha)\langle \phi_0,h^\sharp\beta\rangle)\gamma+\langle \phi_0,h^\sharp\beta\rangle(\alpha\cdot^{h^\sharp}\gamma)\\
&\qquad -(\rho_A(h^\sharp\alpha)(h(\beta,\gamma)))\phi_0-h(\beta,\gamma)(\alpha\cdot^{h^\sharp}\phi_0)+\langle\phi_0,h^\sharp\alpha\rangle(\beta \cdot^{h^{\sharp}} \gamma)\\
&\qquad +\langle\phi_0,h^\sharp\alpha\rangle\langle \phi_0,h^\sharp\beta\rangle\gamma-\langle\phi_0,h^\sharp\alpha\rangle h(\beta,\gamma)\phi_0-\langle \alpha,h^\sharp(\beta \cdot^{h^{\sharp}} \gamma)\rangle\phi_0\\
&=\alpha\cdot^{h^\sharp}(\beta \cdot^{h^{\sharp}} \gamma)+(\rho_A(h^\sharp\alpha)(h(\beta,\phi_0)))\gamma+h(\beta, \phi_0)(\alpha\cdot^{h^\sharp}\gamma)\\
&\qquad -(\rho_A(h^\sharp\alpha)(h(\beta,\gamma)))\phi_0-h(\beta,\gamma)(\alpha\cdot^{h^\sharp}\phi_0)\\
&\qquad +h(\alpha,\phi_0)(\beta \cdot^{h^{\sharp}} \gamma)+h(\alpha,\phi_0)h(\beta,\phi_0)\gamma-h(\alpha,\phi_0)h(\beta,\gamma)\phi_0\\
&\qquad -(\llbracket h,h \rrbracket_A^{\phi_0}(\beta,\alpha,\gamma)+\langle \alpha,h^\sharp\beta \cdot_A h^\sharp\gamma\rangle)\phi_0\\
&=\alpha\cdot^{h^\sharp}(\beta \cdot^{h^{\sharp}} \gamma)+(\rho_A(h^\sharp\alpha)(h(\beta,\phi_0)))\gamma+h(\beta, \phi_0)(\alpha\cdot^{h^\sharp}\gamma)\\
&\qquad -(\rho_A(h^\sharp\alpha)(h(\beta,\gamma)))\phi_0-h(\beta,\gamma)(\alpha\cdot^{h^\sharp}\phi_0)\\
&\qquad +h(\alpha,\phi_0)(\beta \cdot^{h^{\sharp}} \gamma)+h(\alpha,\phi_0)h(\beta,\phi_0)\gamma-h(\alpha,\phi_0)h(\beta,\gamma)\phi_0\\
&\qquad -\llbracket h,h \rrbracket_A^{\phi_0}(\beta,\alpha,\gamma)\phi_0-\langle \alpha,h^\sharp\beta \cdot_A h^\sharp\gamma\rangle\phi_0,\\
\beta&\cdot^{h^\sharp,\phi_0}(\alpha\cdot^{h^\sharp,\phi_0}\gamma)\\
&=\beta\cdot^{h^\sharp}(\alpha \cdot^{h^{\sharp}} \gamma)+(\rho_A(h^\sharp\beta)(h(\alpha,\phi_0)))\gamma+h(\alpha, \phi_0)(\beta\cdot^{h^\sharp}\gamma)\\
&\qquad -(\rho_A(h^\sharp\beta)(h(\alpha,\gamma)))\phi_0-h(\alpha,\gamma)(\beta\cdot^{h^\sharp}\phi_0)\\
&\qquad +h(\beta,\phi_0)(\alpha \cdot^{h^{\sharp}} \gamma)+h(\beta,\phi_0)h(\alpha,\phi_0)\gamma-h(\beta,\phi_0)h(\alpha,\gamma)\phi_0\\
&\qquad -\llbracket h,h \rrbracket_A^{\phi_0}(\alpha,\beta,\gamma)\phi_0-\langle \beta,h^\sharp\alpha \cdot_A h^\sharp\gamma\rangle\phi_0.
\end{align*}
Therefore it follows that
\begin{align*}
(\alpha&\cdot^{h^\sharp,\phi_0}\beta)\cdot^{h^\sharp,\phi_0}\gamma-\alpha\cdot^{h^\sharp,\phi_0}(\beta\cdot^{h^\sharp,\phi_0}\gamma)\\
&\qquad -((\beta\cdot^{h^\sharp,\phi_0}\alpha)\cdot^{h^\sharp,\phi_0}\gamma-\beta\cdot^{h^\sharp,\phi_0}(\alpha\cdot^{h^\sharp,\phi_0}\gamma))\\
&=(\alpha\cdot^{h^\sharp}\beta)\cdot^{h^\sharp}\gamma-\alpha\cdot^{h^\sharp}(\beta\cdot^{h^\sharp}\gamma)-((\beta\cdot^{h^\sharp}\alpha)\cdot^{h^\sharp}\gamma-\beta\cdot^{h^\sharp}(\alpha\cdot^{h^\sharp}\gamma))\\
&\qquad +h(\alpha,\phi_0)(\beta\cdot^{h^\sharp}\gamma)
-h(\beta,\phi_0)(\alpha\cdot^{h^\sharp}\gamma)\\
&\qquad +h(\beta,\gamma)(\alpha\cdot^{h^\sharp}\phi_0)-h(\alpha,\gamma)(\beta\cdot^{h^\sharp}\phi_0)\\
&\qquad -\llbracket h,h\rrbracket_{A}^{\phi_0}(\alpha,\gamma,\beta)\phi_0+\llbracket h,h\rrbracket_{A}^{\phi_0}(\beta,\gamma,\alpha)\phi_0\\
&\qquad +\llbracket h,h\rrbracket_{A}^{\phi_0}(\beta,\alpha,\gamma)\phi_0-\llbracket h,h\rrbracket_{A}^{\phi_0}(\alpha,\beta,\gamma)\phi_0\\
&\qquad +(\rho_A(h^\sharp\alpha)(h(\beta,\gamma))-\rho_A(h^\sharp\beta)(h(\alpha,\gamma))+\langle \alpha,h^\sharp\beta\cdot_Ah^\sharp\gamma\rangle\\
&\qquad -\langle \beta,h^\sharp\alpha\cdot_Ah^\sharp\gamma\rangle
-\langle\gamma,h^\sharp\alpha\cdot_A h^\sharp\beta-h^\sharp\beta\cdot_A h^\sharp\alpha\rangle\\
&\qquad +h(\phi_0,\alpha)h(\beta,\gamma)-h(\phi_0,\beta)h(\gamma,\alpha))\phi_0\\
&\qquad +(\rho_A(h^\sharp\beta)\langle\phi_0,h^\sharp\alpha\rangle-\langle\phi_0,h^\sharp\beta\cdot_Ah^\sharp\alpha\rangle)\gamma\\
&\qquad -(\rho_A(h^\sharp\alpha)\langle\phi_0,h^\sharp\beta\rangle-\langle\phi_0,h^\sharp\alpha\cdot_Ah^\sharp\beta\rangle)\gamma\\
&\qquad +\llbracket h,h\rrbracket_{A}^{\phi_0}(\alpha,\phi_0,\beta)\gamma-\llbracket h,h\rrbracket_{A}^{\phi_0}(\beta,\phi_0,\alpha)\gamma\\
&=(\alpha\cdot^{h^\sharp}\beta)\cdot^{h^\sharp}\gamma-\alpha\cdot^{h^\sharp}(\beta\cdot^{h^\sharp}\gamma)-((\beta\cdot^{h^\sharp}\alpha)\cdot^{h^\sharp}\gamma-\beta\cdot^{h^\sharp}(\alpha\cdot^{h^\sharp}\gamma))\\
&\qquad +h(\alpha,\phi_0)(\beta\cdot^{h^\sharp}\gamma)
-h(\beta,\phi_0)(\alpha\cdot^{h^\sharp}\gamma)\\
&\qquad +h(\beta,\gamma)(\alpha\cdot^{h^\sharp}\phi_0)-h(\alpha,\gamma)(\beta\cdot^{h^\sharp}\phi_0)\\
&\qquad -\llbracket h,h\rrbracket_{A}^{\phi_0}(\alpha,\gamma,\beta)\phi_0+\llbracket h,h\rrbracket_{A}^{\phi_0}(\beta,\gamma,\alpha)\phi_0\\
&\qquad +\llbracket h,h\rrbracket_{A}^{\phi_0}(\beta,\alpha,\gamma)\phi_0-\llbracket h,h\rrbracket_{A}^{\phi_0}(\alpha,\beta,\gamma)\phi_0\\
&\qquad +(\rho_A(h^\sharp\alpha)(h(\beta,\gamma))-\rho_A(h^\sharp\beta)(h(\alpha,\gamma))+\langle \alpha,h^\sharp\beta\cdot_Ah^\sharp\gamma\rangle\\
&\qquad -\langle \beta,h^\sharp\alpha\cdot_Ah^\sharp\gamma\rangle
-\langle\gamma,[h^\sharp\alpha, h^\sharp\beta]_A\rangle\\
&\qquad +h(\phi_0,\alpha)h(\beta,\gamma)-h(\phi_0,\beta)h(\gamma,\alpha))\phi_0\\
&\qquad -(\delta_A\phi_0)(h^\sharp\alpha,h^\sharp\beta)\gamma+(\delta_A\phi_0)(h^\sharp\beta,h^\sharp\alpha)\gamma\\
&\qquad +\llbracket h,h\rrbracket_{A}^{\phi_0}(\alpha,\phi_0,\beta)\gamma-\llbracket h,h\rrbracket_{A}^{\phi_0}(\beta,\phi_0,\alpha)\gamma\\
&=(\alpha\cdot^{h^\sharp}\beta)\cdot^{h^\sharp}\gamma-\alpha\cdot^{h^\sharp}(\beta\cdot^{h^\sharp}\gamma)-((\beta\cdot^{h^\sharp}\alpha)\cdot^{h^\sharp}\gamma-\beta\cdot^{h^\sharp}(\alpha\cdot^{h^\sharp}\gamma))\\
&\qquad +h(\alpha,\phi_0)(\beta\cdot^{h^\sharp}\gamma)
-h(\beta,\phi_0)(\alpha\cdot^{h^\sharp}\gamma)\\
&\qquad +h(\beta,\gamma)(\alpha\cdot^{h^\sharp}\phi_0)-h(\alpha,\gamma)(\beta\cdot^{h^\sharp}\phi_0)\\
&\qquad -\llbracket h,h\rrbracket_{A}^{\phi_0}(\alpha,\gamma,\beta)\phi_0+\llbracket h,h\rrbracket_{A}^{\phi_0}(\beta,\gamma,\alpha)\phi_0\\
&\qquad +\llbracket h,h\rrbracket_{A}^{\phi_0}(\beta,\alpha,\gamma)\phi_0-\llbracket h,h\rrbracket_{A}^{\phi_0}(\alpha,\beta,\gamma)\phi_0\\
&\qquad +\llbracket h,h \rrbracket_A^{\phi_0}(\alpha,\beta,\gamma)\phi_0-((\delta_A\phi_0)(h^\sharp\alpha,h^\sharp\beta)-(\delta_A\phi_0)(h^\sharp\beta,h^\sharp\alpha))\gamma\\
&\qquad +\llbracket h,h\rrbracket_{A}^{\phi_0}(\alpha,\phi_0,\beta)\gamma-\llbracket h,h\rrbracket_{A}^{\phi_0}(\beta,\phi_0,\alpha)\gamma\\
&=(\alpha\cdot^{h^\sharp}\beta)\cdot^{h^\sharp}\gamma-\alpha\cdot^{h^\sharp}(\beta\cdot^{h^\sharp}\gamma)-((\beta\cdot^{h^\sharp}\alpha)\cdot^{h^\sharp}\gamma-\beta\cdot^{h^\sharp}(\alpha\cdot^{h^\sharp}\gamma))\\
&\qquad +h(\alpha,\phi_0)(\beta\cdot^{h^\sharp}\gamma)
-h(\beta,\phi_0)(\alpha\cdot^{h^\sharp}\gamma)\\
&\qquad +h(\beta,\gamma)(\alpha\cdot^{h^\sharp}\phi_0)-h(\alpha,\gamma)(\beta\cdot^{h^\sharp}\phi_0)\\
&\qquad -\llbracket h,h\rrbracket_{A}^{\phi_0}(\alpha,\gamma,\beta)\phi_0+\llbracket h,h\rrbracket_{A}^{\phi_0}(\beta,\gamma,\alpha)\phi_0\\
&\qquad +\llbracket h,h\rrbracket_{A}^{\phi_0}(\beta,\alpha,\gamma)\phi_0-((\delta_A\phi_0)(h^\sharp\alpha,h^\sharp\beta)-(\delta_A\phi_0)(h^\sharp\beta,h^\sharp\alpha))\gamma\\
&\qquad +\llbracket h,h\rrbracket_{A}^{\phi_0}(\alpha,\phi_0,\beta)\gamma-\llbracket h,h\rrbracket_{A}^{\phi_0}(\beta,\phi_0,\alpha)\gamma.
\end{align*}
By using that $\llbracket h,h \rrbracket_A^{\phi_0}=0$ and symmetry of $\delta_A\phi_0$, we obtain
\begin{align*}
(\alpha&\cdot^{h^\sharp,\phi_0}\beta)\cdot^{h^\sharp,\phi_0}\gamma-\alpha\cdot^{h^\sharp,\phi_0}(\beta\cdot^{h^\sharp,\phi_0}\gamma)\\
&\qquad -((\beta\cdot^{h^\sharp,\phi_0}\alpha)\cdot^{h^\sharp,\phi_0}\gamma-\beta\cdot^{h^\sharp,\phi_0}(\alpha\cdot^{h^\sharp,\phi_0}\gamma))\\
&=(\alpha\cdot^{h^\sharp}\beta)\cdot^{h^\sharp}\gamma-\alpha\cdot^{h^\sharp}(\beta\cdot^{h^\sharp}\gamma)-((\beta\cdot^{h^\sharp}\alpha)\cdot^{h^\sharp}\gamma-\beta\cdot^{h^\sharp}(\alpha\cdot^{h^\sharp}\gamma))\\
&\qquad +h(\alpha,\phi_0)(\beta\cdot^{h^\sharp}\gamma)
-h(\beta,\phi_0)(\alpha\cdot^{h^\sharp}\gamma)\\
&\qquad +h(\beta,\gamma)(\alpha\cdot^{h^\sharp}\phi_0)-h(\alpha,\gamma)(\beta\cdot^{h^\sharp}\phi_0).
\end{align*}
Hence, for any $X$ in $\Gamma(A)$, it follows from (\ref{left-symmetry of cdot h sharp}) and the definition of $\llbracket h,h \rrbracket_A^{\phi_0}$ that
\begin{align*}
\langle (\alpha&\cdot^{h^\sharp,\phi_0}\beta)\cdot^{h^\sharp,\phi_0}\gamma-\alpha\cdot^{h^\sharp,\phi_0}(\beta\cdot^{h^\sharp,\phi_0}\gamma)\\
&\qquad -((\beta\cdot^{h^\sharp,\phi_0}\alpha)\cdot^{h^\sharp,\phi_0}\gamma-\beta\cdot^{h^\sharp,\phi_0}(\alpha\cdot^{h^\sharp,\phi_0}\gamma)),X\rangle\\
&=\langle (\alpha\cdot^{h^\sharp}\beta)\cdot^{h^\sharp}\gamma-\alpha\cdot^{h^\sharp}(\beta\cdot^{h^\sharp}\gamma)-((\beta\cdot^{h^\sharp}\alpha)\cdot^{h^\sharp}\gamma-\beta\cdot^{h^\sharp}(\alpha\cdot^{h^\sharp}\gamma)),X\rangle\\
&\qquad +h(\alpha,\phi_0)\langle\beta\cdot^{h^\sharp}\gamma,X\rangle
-h(\beta,\phi_0)\langle\alpha\cdot^{h^\sharp}\gamma,X\rangle\\
&\qquad +h(\beta,\gamma)\langle\alpha\cdot^{h^\sharp}\phi_0,X\rangle-h(\alpha,\gamma)\langle\beta\cdot^{h^\sharp}\phi_0,X\rangle\\
&=\langle\mathcal{L}_{\llbracket h,h\rrbracket_A(\alpha,\,\cdot\,,\beta)-\llbracket h,h\rrbracket_A(\beta,\,\cdot\,,\alpha)}^A\gamma,X\rangle\\
&\qquad +\llbracket h,h\rrbracket_A(\beta,L_X\alpha,\gamma)-\llbracket h,h\rrbracket_A(\alpha,L_X\beta,\gamma)\\
&\qquad +h(\alpha,\phi_0)\langle\mathcal{L}_{h^\sharp\beta}^A\gamma-R_{h^\sharp\gamma}\beta-d_A(h(\beta,\gamma)),X\rangle\\
&\qquad -h(\beta,\phi_0)\langle\mathcal{L}_{h^\sharp\alpha}^A\gamma-R_{h^\sharp\gamma}\alpha-d_A(h(\alpha,\gamma)),X\rangle\\
&\qquad +h(\beta,\gamma)\langle\mathcal{L}_{h^\sharp\alpha}^A\phi_0-R_{h^\sharp\phi_0}\alpha-d_A(h(\alpha,\phi_0)),X\rangle\\
&\qquad -h(\alpha,\gamma)\langle\mathcal{L}_{h^\sharp\beta}^A\phi_0-R_{h^\sharp\phi_0}\beta-d_A(h(\beta,\phi_0)),X\rangle\\
&=\langle\mathcal{L}_{\llbracket h,h\rrbracket_A(\alpha,\,\cdot\,,\beta)}^A\gamma,X\rangle-\langle\mathcal{L}_{\llbracket h,h\rrbracket_A(\beta,\,\cdot\,,\alpha)}^A\gamma,X\rangle\\
&\qquad +\llbracket h,h\rrbracket_A(\beta,L_X\alpha,\gamma)-\llbracket h,h\rrbracket_A(\alpha,L_X\beta,\gamma)\\
&\qquad +h(\alpha,\phi_0)(\rho_A(h^\sharp\beta)\langle\gamma,X\rangle-\langle\gamma,\mathcal{L}_{h^\sharp\beta}^AX\rangle\\
&\qquad \phantom{+h(\alpha,\phi_0)(\rho_A(h^\sharp\beta)\langle\gamma,X\rangle}+\langle\beta,X\cdot_Ah^\sharp\gamma\rangle-\rho_A(X)(h(\beta,\gamma)))\\
&\qquad -h(\beta,\phi_0)(\rho_A(h^\sharp\alpha)\langle\gamma,X\rangle-\langle\gamma,\mathcal{L}_{h^\sharp\alpha}^AX\rangle\\
&\qquad \phantom{+h(\alpha,\phi_0)(\rho_A(h^\sharp\alpha)\langle\gamma,X\rangle}+\langle\alpha,X\cdot_Ah^\sharp\gamma\rangle-\rho_A(X)(h(\alpha,\gamma)))\\
&\qquad +h(\beta,\gamma)(\rho_A(h^\sharp\alpha)\langle\phi_0,X\rangle-\langle\phi_0,\mathcal{L}_{h^\sharp\alpha}^AX\rangle\\
&\qquad \phantom{+h(\alpha,\phi_0)(\rho_A(h^\sharp\alpha)\langle\gamma,X\rangle}+\langle\alpha,X\cdot_Ah^\sharp\phi_0\rangle-\rho_A(X)(h(\alpha,\phi_0)))\\
&\qquad -h(\alpha,\gamma)(\rho_A(h^\sharp\beta)\langle\phi_0,X\rangle-\langle\phi_0,\mathcal{L}_{h^\sharp\beta}^AX\rangle\\
&\qquad \phantom{+h(\alpha,\phi_0)(\rho_A(h^\sharp\alpha)\langle\gamma,X\rangle}+\langle\beta,X\cdot_Ah^\sharp\phi_0\rangle-\rho_A(X)(h(\beta,\phi_0)))\\
&=\langle\mathcal{L}_{\llbracket h,h\rrbracket_A(\alpha,\,\cdot\,,\beta)}^A\gamma,X\rangle-\langle\mathcal{L}_{\llbracket h,h\rrbracket_A(\beta,\,\cdot\,,\alpha)}^A\gamma,X\rangle\\
&\qquad +\llbracket h,h\rrbracket_A(\beta,L_X\alpha,\gamma)-\llbracket h,h\rrbracket_A(\alpha,L_X\beta,\gamma)\\
&\qquad +\rho_A(h(\alpha,\phi_0)h^\sharp\beta)\langle\gamma,X\rangle-h(\alpha,\phi_0)\langle\gamma,[h^\sharp\beta,X]_A\rangle\\
&\qquad \phantom{+\rho_A(h(\alpha,\phi_0)}+h(\alpha,\phi_0)(\langle\beta,X\cdot_Ah^\sharp\gamma\rangle-\rho_A(X)\langle h^\sharp\beta,\gamma\rangle)\\
&\qquad -\rho_A(h(\beta,\phi_0)h^\sharp\alpha)\langle\gamma,X\rangle+h(\beta,\phi_0)\langle\gamma,[h^\sharp\alpha,X]_A\rangle\\
&\qquad \phantom{+\rho_A(h(\alpha,\phi_0)}-h(\beta,\phi_0)(\langle\alpha,X\cdot_Ah^\sharp\gamma\rangle-\rho_A(X)\langle h^\sharp\alpha,\gamma\rangle)\\
&\qquad +h(\beta,\gamma)(\rho_A(h^\sharp\alpha)\langle\phi_0,X\rangle-\langle\phi_0,[h^\sharp\alpha,X]_A\rangle)\\
&\qquad \phantom{+\rho_A(h(\alpha,\gamma)}+h(\beta,\gamma)(\langle\alpha,X\cdot_Ah^\sharp\phi_0\rangle-\rho_A(X)\langle \alpha,h^\sharp\phi_0\rangle)\\
&\qquad -h(\alpha,\gamma)(\rho_A(h^\sharp\beta)\langle\phi_0,X\rangle-\langle\phi_0,[h^\sharp\beta,X]_A)\rangle\\
&\qquad \phantom{+\rho_A(h(\beta,\gamma)}-h(\alpha,\gamma)(\langle\beta,X\cdot_Ah^\sharp\phi_0)\rangle-\rho_A(X)\langle \beta,h^\sharp\phi_0\rangle)\\
&=\langle\mathcal{L}_{\llbracket h,h\rrbracket_A(\alpha,\,\cdot\,,\beta)}^A\gamma,X\rangle-\langle\mathcal{L}_{\llbracket h,h\rrbracket_A(\beta,\,\cdot\,,\alpha)}^A\gamma,X\rangle\\
&\qquad +\llbracket h,h\rrbracket_A(\beta,L_X\alpha,\gamma)-\llbracket h,h\rrbracket_A(\alpha,L_X\beta,\gamma)\\
&\qquad +\rho_A(h(\alpha,\phi_0)h^\sharp\beta)\langle\gamma,X\rangle-\langle\gamma,h(\alpha,\phi_0)[h^\sharp\beta,X]_A\rangle\\
&\qquad \phantom{+\rho_A(h(\alpha,\phi_0)h^\sharp\beta)\langle\gamma,X\rangle-h(\alpha,\phi_0)}-h(\alpha,\phi_0)\langle L_X\beta,h^\sharp\gamma\rangle\\
&\qquad -\rho_A(h(\beta,\phi_0)h^\sharp\alpha)\langle\gamma,X\rangle+\langle\gamma,h(\beta,\phi_0)[h^\sharp\alpha,X]_A\rangle\\
&\qquad \phantom{+\rho_A(h(\beta,\phi_0)h^\sharp\alpha)\langle\gamma,X\rangle-h(\beta,\phi_0)}+h(\beta,\phi_0)\langle L_X\alpha,h^\sharp\gamma\rangle\\
&\qquad +h(\beta,\gamma)((d_A\phi_0)(h^\sharp\alpha,X)+\rho_A(X)\langle \phi_0,h^\sharp\alpha\rangle)-h(\beta,\gamma)\langle L_X\alpha,h^\sharp\phi_0\rangle\\
&\qquad -h(\alpha,\gamma)((d_A\phi_0)(h^\sharp\beta,X)+\rho_A(X)\langle \phi_0,h^\sharp\beta\rangle)+h(\alpha,\gamma)\langle L_X\beta,h^\sharp\phi_0\rangle\\
&=\langle\mathcal{L}_{\llbracket h,h\rrbracket_A(\alpha,\,\cdot\,,\beta)}^A\gamma,X\rangle-\langle\mathcal{L}_{\llbracket h,h\rrbracket_A(\beta,\,\cdot\,,\alpha)}^A\gamma,X\rangle\\
&\qquad +(\llbracket h,h\rrbracket_A(\beta,L_X\alpha,\gamma)+h(\phi_0,\beta)h(L_X\alpha,\gamma)-h(\phi_0,L_X\alpha)h(\beta,\gamma))\\
&\qquad -(\llbracket h,h\rrbracket_A(\alpha,L_X\beta,\gamma)+h(\phi_0,\alpha)h( L_X\beta,\gamma)-h(\phi_0,L_X\beta)h(\alpha,\gamma))\\
&\qquad +\rho_A(h(\alpha,\phi_0)h^\sharp\beta)\langle\gamma,X\rangle-\langle\gamma,h(\alpha,\phi_0)[h^\sharp\beta,X]_A\rangle\\
&\qquad -\rho_A(h(\beta,\phi_0)h^\sharp\alpha)\langle\gamma,X\rangle+\langle\gamma,h(\beta,\phi_0)[h^\sharp\alpha,X]_A\rangle\\
&\qquad +h(\beta,\gamma)(d_A\phi_0)(h^\sharp\alpha,X)+\langle(\rho_A(X)(h(\alpha,\phi_0)))h^\sharp\beta,\gamma\rangle\\
&\qquad -h(\alpha,\gamma)(d_A\phi_0)(h^\sharp\beta,X)+\langle(\rho_A(X)(h(\beta,\phi_0)))h^\sharp\alpha,\gamma\rangle\\
&=\langle\mathcal{L}_{\llbracket h,h\rrbracket_A(\alpha,\,\cdot\,,\beta)}^A\gamma,X\rangle-\langle\mathcal{L}_{\llbracket h,h\rrbracket_A(\beta,\,\cdot\,,\alpha)}^A\gamma,X\rangle\\
&\qquad +\llbracket h,h\rrbracket_A^{\phi_0}(\beta,L_X\alpha,\gamma)-\llbracket h,h\rrbracket_A^{\phi_0}(\alpha,L_X\beta,\gamma)\\
&\qquad +\rho_A(h(\alpha,\phi_0)h^\sharp\beta)\langle\gamma,X\rangle\\
&\qquad \phantom{+\rho_A(h(\alpha,\phi_0)}-\langle\gamma,h(\alpha,\phi_0)[h^\sharp\beta,X]_A-(\rho_A(X)(h(\alpha,\phi_0)))h^\sharp\beta\rangle\\
&\qquad -(\rho_A(h(\beta,\phi_0)h^\sharp\alpha)\langle\gamma,X\rangle\\
&\qquad \phantom{+\rho_A(h(\beta,\phi_0)}-\langle\gamma,h(\beta,\phi_0)[h^\sharp\alpha,X]_A-(\rho_A(X)(h(\beta,\phi_0)))h^\sharp\alpha\rangle)\\
&\qquad +h(\beta,\gamma)(d_A\phi_0)(h^\sharp\alpha,X)-h(\alpha,\gamma)(d_A\phi_0)(h^\sharp\beta,X)\\
&=\langle\mathcal{L}_{\llbracket h,h\rrbracket_A(\alpha,\,\cdot\,,\beta)}^A\gamma,X\rangle-\langle\mathcal{L}_{\llbracket h,h\rrbracket_A(\beta,\,\cdot\,,\alpha)}^A\gamma,X\rangle\\
&\qquad +\llbracket h,h\rrbracket_A^{\phi_0}(\beta,L_X\alpha,\gamma)-\llbracket h,h\rrbracket_A^{\phi_0}(\alpha,L_X\beta,\gamma)\\
&\qquad +\rho_A(h(\alpha,\phi_0)h^\sharp\beta)\langle\gamma,X\rangle-\langle\gamma,[h(\alpha,\phi_0)h^\sharp\beta,X]_A\rangle\\
&\qquad -(\rho_A(h(\beta,\phi_0)h^\sharp\alpha)\langle\gamma,X\rangle-\langle\gamma,[h(\beta,\phi_0)h^\sharp\alpha,X]_A\rangle\\
&\qquad +h(\beta,\gamma)(d_A\phi_0)(h^\sharp\alpha,X)-h(\alpha,\gamma)(d_A\phi_0)(h^\sharp\beta,X)\\
&=\langle\mathcal{L}_{\llbracket h,h\rrbracket_A(\alpha,\,\cdot\,,\beta)}^A\gamma,X\rangle+\rho_A(h(\alpha,\phi_0)h^\sharp\beta)\langle\gamma,X\rangle-\langle\gamma,\mathcal{L}_{h(\alpha,\phi_0)h^\sharp\beta}^AX\rangle\\
&\qquad -(\langle\mathcal{L}_{\llbracket h,h\rrbracket_A(\beta,\,\cdot\,,\alpha)}^A\gamma,X\rangle+\rho_A(h(\beta,\phi_0)h^\sharp\alpha)\langle\gamma,X\rangle-\langle\gamma,\mathcal{L}_{h(\beta,\phi_0)h^\sharp\alpha}^AX\rangle)\\
&\qquad +\llbracket h,h\rrbracket_A^{\phi_0}(\beta,L_X\alpha,\gamma)-\llbracket h,h\rrbracket_A^{\phi_0}(\alpha,L_X\beta,\gamma)\\
&\qquad +h(\beta,\gamma)(d_A\phi_0)(h^\sharp\alpha,X)-h(\alpha,\gamma)(d_A\phi_0)(h^\sharp\beta,X)\\
&=\langle\mathcal{L}_{\llbracket h,h\rrbracket_A(\alpha,\,\cdot\,,\beta)}^A\gamma,X\rangle+\langle\mathcal{L}_{h(\alpha,\phi_0)h^\sharp\beta}^A\gamma,X\rangle-\langle\mathcal{L}_{h(\beta,\alpha)h^\sharp\phi_0}^A\gamma,X\rangle\\
&\qquad -(\langle\mathcal{L}_{\llbracket h,h\rrbracket_A(\beta,\,\cdot\,,\alpha)}^A\gamma,X\rangle+\langle\mathcal{L}_{h(\beta,\phi_0)h^\sharp\alpha}^A\gamma,X\rangle-\langle\mathcal{L}_{h(\alpha,\beta)h^\sharp\phi_0}^A\gamma,X\rangle)\\
&\qquad +\llbracket h,h\rrbracket_A^{\phi_0}(\beta,L_X\alpha,\gamma)-\llbracket h,h\rrbracket_A^{\phi_0}(\alpha,L_X\beta,\gamma)\\
&\qquad +h(\beta,\gamma)(d_A\phi_0)(h^\sharp\alpha,X)-h(\alpha,\gamma)(d_A\phi_0)(h^\sharp\beta,X)\\
&=\langle\mathcal{L}_{\llbracket h,h\rrbracket_A(\alpha,\,\cdot\,,\beta)+h(\alpha,\phi_0)h^\sharp\beta-h(\beta,\alpha)h^\sharp\phi_0}^A\gamma,X\rangle\\
&\qquad -\langle\mathcal{L}_{\llbracket h,h\rrbracket_A(\beta,\,\cdot\,,\alpha)+h(\beta,\phi_0)h^\sharp\alpha-h(\alpha,\beta)h^\sharp\phi_0}^A\gamma,X\rangle\\
&\qquad +\llbracket h,h\rrbracket_A^{\phi_0}(\beta,L_X\alpha,\gamma)-\llbracket h,h\rrbracket_A^{\phi_0}(\alpha,L_X\beta,\gamma)\\
&\qquad +h(\beta,\gamma)(d_A\phi_0)(h^\sharp\alpha,X)-h(\alpha,\gamma)(d_A\phi_0)(h^\sharp\beta,X)\\
&=\left\langle\mathcal{L}_{\llbracket h,h\rrbracket_A^{\phi_0}(\alpha,\,\cdot\,,\beta)}^A\gamma,X\right\rangle-\left\langle\mathcal{L}_{\llbracket h,h\rrbracket_A^{\phi_0}(\beta,\,\cdot\,,\alpha)}^A\gamma,X\right\rangle\\
&\qquad +\llbracket h,h\rrbracket_A^{\phi_0}(\beta,L_X\alpha,\gamma)-\llbracket h,h\rrbracket_A^{\phi_0}(\alpha,L_X\beta,\gamma)\\
&\qquad +h(\beta,\gamma)(d_A\phi_0)(h^\sharp\alpha,X)-h(\alpha,\gamma)(d_A\phi_0)(h^\sharp\beta,X).
\end{align*}
By using that $\llbracket h,h \rrbracket_A^{\phi_0}=0$ and the $d_A$-closedness of $\phi_0$, we obtain
\begin{align*}
\langle (&\alpha\cdot^{h^\sharp,\phi_0}\beta)\cdot^{h^\sharp,\phi_0}\gamma-\alpha\cdot^{h^\sharp,\phi_0}(\beta\cdot^{h^\sharp,\phi_0}\gamma)\\
&\phantom{\alpha} -((\beta\cdot^{h^\sharp,\phi_0}\alpha)\cdot^{h^\sharp,\phi_0}\gamma-\beta\cdot^{h^\sharp,\phi_0}(\alpha\cdot^{h^\sharp,\phi_0}\gamma)),X\rangle=0,\\
&\alpha\cdot^{h^\sharp,\phi_0}\beta)\cdot^{h^\sharp,\phi_0}\gamma-\alpha\cdot^{h^\sharp,\phi_0}(\beta\cdot^{h^\sharp,\phi_0}\gamma)\\
&\phantom{\alpha} -((\beta\cdot^{h^\sharp,\phi_0}\alpha)\cdot^{h^\sharp,\phi_0}\gamma-\beta\cdot^{h^\sharp,\phi_0}(\alpha\cdot^{h^\sharp,\phi_0}\gamma))=0.
\end{align*}
Therefore $A^*_{h,\phi_0}$ is a left-symmetric algebroid over $M$. Finally, we prove symmetry of $\delta_{h,\phi_0}(-h^\sharp\phi_0)$, where $\delta_{h,\phi_0}$ is the coboundary operator of the left-symmetric algebroid $A_{h,\phi_0}^*$. It follows that, for any $\alpha, \beta$ in $\Gamma(A^*)$,
\begin{align*}
(\delta_{h,\phi_0}&(-h^\sharp\phi_0))(\alpha, \beta)\\
&=\rho_A(h^\sharp\alpha)\langle-h^\sharp\phi_0,\beta\rangle-\langle-h^\sharp\phi_0,\alpha\cdot^{h^\sharp,\phi_0}\beta\rangle\\
&=-\rho_A(h^\sharp\alpha)\langle\phi_0,h^\sharp\beta\rangle+\langle\phi_0,h^\sharp(\alpha\cdot^{h^\sharp,\phi_0}\beta)\rangle\\
&=-(\rho_A(h^\sharp\alpha)\langle\phi_0,h^\sharp\beta\rangle-(\llbracket h,h\rrbracket_{A}^{\phi_0}(h^\sharp\alpha,\phi_0,h^\sharp\beta)+\langle\phi_0,h^\sharp\alpha\cdot_Ah^\sharp\beta\rangle))\\
&=-(\rho_A(h^\sharp\alpha)\langle\phi_0,h^\sharp\beta\rangle-\langle\phi_0,h^\sharp\alpha\cdot_Ah^\sharp\beta\rangle)-\llbracket h,h\rrbracket_{A}^{\phi_0}(h^\sharp\alpha,\phi_0,h^\sharp\beta)\\
&=-(\delta_A\phi_0)(h^\sharp\alpha,h^\sharp\beta)-\llbracket h,h\rrbracket_{A}^{\phi_0}(h^\sharp\alpha,\phi_0,h^\sharp\beta),\\
(\delta_{h,\phi_0}&(-h^\sharp\phi_0))(\beta,\alpha)=-(\delta_A\phi_0)(h^\sharp\beta,h^\sharp\alpha)-\llbracket h,h\rrbracket_{A}^{\phi_0}(h^\sharp\beta,\phi_0,h^\sharp\alpha).
\end{align*}
By using that $\llbracket h,h \rrbracket_A^{\phi_0}=0$ and symmetry of $\delta_A\phi_0$, we obtain
\begin{align*}
(\delta_{h,\phi_0}(-h^\sharp\phi_0))(\alpha, \beta)=(\delta_{h,\phi_0}(-h^\sharp\phi_0))(\beta,\alpha).
\end{align*}
Therefore $(A^*_{h,\phi_0}, -h^{\sharp}\phi_0)$ is a Jacobi-left-symmetric algebroid over $M$.
\end{proof}

\begin{prop}\label{Jacobi-lsa to lsa}
Let $(A,\phi_0)$ be a Jacobi-left-symmetric algebroid over $M$.
Then $(\tilde{A},\hat{\cdot}_A^{\phi_0},\hat{\rho}_A^{\phi_0})$ and $(\tilde{A},\bar{\cdot}_A^{\phi_0},\bar{\rho}_A^{\phi_0})$ are left-symmetric algebroids over $M\times \mathbb{R}$, where
$\hat{\cdot}_A^{\phi_0}$ and $\bar{\cdot}_A^{\phi_0}$ are defined by, for any $\tilde{X},\tilde{Y}$ in $\Gamma(\tilde{A})$,
\begin{align}
\label{hat dot} \tilde{X} \:\hat{\cdot}_A^{\phi_0} \:\tilde{Y}&:=e^{-t}\left(\tilde{X}\cdot_A\tilde{Y}+\langle\phi_0,\tilde{X}\rangle\left(\frac{\partial \tilde{Y}}{\partial t}-\tilde{Y}\right)\right),\\ 
\label{bar dot} \tilde{X} \:\bar{\cdot}_A^{\phi_0} \:\tilde{Y} &:=\tilde{X}\cdot_A \tilde{Y}+\phi_0(\tilde{X})\frac{\partial \tilde{Y}}{\partial t},
\end{align}
and $\hat{\rho}_A^{\phi_0}$ and $\bar{\rho}_A^{\phi_0}$ are defined by (\ref{hat anchor}) and (\ref{bar anchor}) respectively. Conversely, for a left-symmetric algebroid $A$ over $M$ and $\phi_0$ in $\Gamma(A^*)$, if the triple $(\tilde{A},\hat{\cdot}_A^{\phi_0},\hat{\rho}_A^{\phi_0})$ (resp. $(\tilde{A},\bar{\cdot}_A^{\phi_0},\bar{\rho}_A^{\phi_0})$) defined by (\ref{hat dot}) and (\ref{hat anchor}) (resp. (\ref{bar dot}) and (\ref{bar anchor})) is a left-symmetric algebroid over $M\times \mathbb{R}$, then $(A,\phi_0)$ is a Jacobi-left-symmetric algebroid over $M$, i.e., $\delta_{A}\phi_0$ is symmetric. Furthermore the sub-adjacent Lie algebroid of $(\tilde{A},\hat{\cdot}_A^{\phi_0},\hat{\rho}_A^{\phi_0})$ (resp. $(\tilde{A},\bar{\cdot}_A^{\phi_0},\bar{\rho}_A^{\phi_0})$) is $\tilde{A}_{{\phi}_{0}}^{\wedge}$ (resp. $\tilde{A}_{{\phi}_{0}}^{-}$) in Subsection \ref{Jacobi algebroids and Jacobi structures}.
\end{prop}

\begin{proof}
For any $\tilde{X}, \tilde{Y}$ in $\Gamma(\tilde{A})$ and $\tilde{f}$ in $C^\infty(M\times \mathbb{R})$,
\begin{align*}
    \tilde{X} \:\bar{\cdot}_A^{\phi_0} \:(\tilde{f}\tilde{Y})
    &=\tilde{X} \cdot_A (\tilde{f}\tilde{Y})+\phi_0(\tilde{X})\frac{\partial \tilde{f}\tilde{Y}}{\partial t}\\
    &=\tilde{f}(\tilde{X} \cdot_A \tilde{Y})+(\rho_A(\tilde{X})\tilde{f})\tilde{Y}+\phi_0(\tilde{X})\left(\frac{\partial \tilde{f}}{\partial t}\tilde{Y}+\tilde{f}\frac{\partial \tilde{Y}}{\partial t}\right)\\
    &=\tilde{f}\left(\tilde{X} \cdot_A \tilde{Y}+\phi_0(\tilde{X})\frac{\partial \tilde{Y}}{\partial t}\right)+\left(\rho_A(\tilde{X})\tilde{f}+\phi_0(\tilde{X})\frac{\partial \tilde{f}}{\partial t}\right)\tilde{Y}\\
    &=\tilde{f}(\tilde{X} \:\bar{\cdot}_A^{\phi_0} \:\tilde{Y})+(\bar{\rho}_A^{\phi_0}(\tilde{X})\tilde{f})\tilde{Y},\\
    (\tilde{f}\tilde{X}) \:\bar{\cdot}_A^{\phi_0} \:\tilde{Y}
    &=(\tilde{f}\tilde{X}) \cdot_A\tilde{Y}+\phi_0(\tilde{f}\tilde{X})\frac{\partial \tilde{Y}}{\partial t}\\
    &=\tilde{f}(\tilde{X} \cdot_A \tilde{Y})+\tilde{f}\phi_0(\tilde{X})\frac{\partial \tilde{Y}}{\partial t}\\
    &=\tilde{f}\left(\tilde{X} \cdot_A \tilde{Y}+\phi_0(\tilde{X})\frac{\partial \tilde{Y}}{\partial t}\right)\\
    &=\tilde{f}\left(\tilde{X} \:\bar{\cdot}_A^{\phi_0} \:\tilde{Y}\right).
\end{align*}
For any $\tilde{X}, \tilde{Y}$ and $\tilde{Z}$ in $\Gamma(\tilde{A})$,
\begin{align*}
    \tilde{X} \:\bar{\cdot}_A^{\phi_0} \:(\tilde{Y} \:\bar{\cdot}_A^{\phi_0} \:\tilde{Z})
    &=\tilde{X} \:\bar{\cdot}_A^{\phi_0} \:\left(\tilde{Y}\cdot_A \tilde{Z}+\phi_0(\tilde{Y})\frac{\partial \tilde{Z}}{\partial t}\right)\\
    &=\tilde{X} \cdot_A \left(\tilde{Y}\cdot_A \tilde{Z}+\phi_0(\tilde{Y})\frac{\partial \tilde{Z}}{\partial t}\right)\\
    &\qquad +\phi_0(\tilde{X})\frac{\partial }{\partial t}\left(\tilde{Y}\cdot_A \tilde{Z}+\phi_0(\tilde{Y})\frac{\partial \tilde{Z}}{\partial t}\right)\\
    &=\tilde{X} \cdot_A (\tilde{Y}\cdot_A \tilde{Z})+\tilde{X} \cdot_A \left(\phi_0(\tilde{Y})\frac{\partial \tilde{Z}}{\partial t}\right)\\
    &\qquad +\phi_0(\tilde{X})\left(\frac{\partial \tilde{Y}}{\partial t}\cdot_A \tilde{Z}+\tilde{Y}\cdot_A \frac{\partial \tilde{Z}}{\partial t}\right)\\
    &\qquad +\phi_0(\tilde{X})\left(\phi_0\left(\frac{\partial \tilde{Y}}{\partial t}\right)\frac{\partial \tilde{Z}}{\partial t}+\phi_0(\tilde{Y})\frac{\partial^2 \tilde{Z}}{\partial t^2}\right)\\
    &=\tilde{X} \cdot_A (\tilde{Y}\cdot_A \tilde{Z})+\phi_0(\tilde{Y})\left(\tilde{X} \cdot_A \frac{\partial \tilde{Z}}{\partial t}\right)+\rho_A(\tilde{X})(\phi_0(\tilde{Y}))\frac{\partial \tilde{Z}}{\partial t}\\
    &\qquad +\phi_0(\tilde{X})\left(\frac{\partial \tilde{Y}}{\partial t}\cdot_A \tilde{Z}\right)+\phi_0(\tilde{X})\left(\tilde{Y}\cdot_A \frac{\partial \tilde{Z}}{\partial t}\right)\\
    &\qquad +\phi_0(\tilde{X})\phi_0\left(\frac{\partial \tilde{Y}}{\partial t}\right)\frac{\partial \tilde{Z}}{\partial t}+\phi_0(\tilde{X})\phi_0(\tilde{Y})\frac{\partial^2 \tilde{Z}}{\partial t^2},\\
    \tilde{Y} \:\bar{\cdot}_A^{\phi_0} \:(\tilde{X} \:\bar{\cdot}_A^{\phi_0} \:\tilde{Z})
    &=\tilde{Y} \cdot_A (\tilde{X}\cdot_A \tilde{Z})+\phi_0(\tilde{X})\left(\tilde{Y} \cdot_A \frac{\partial \tilde{Z}}{\partial t}\right)+\rho_A(\tilde{Y})(\phi_0(\tilde{X}))\frac{\partial \tilde{Z}}{\partial t}\\
    &\qquad +\phi_0(\tilde{Y})\left(\frac{\partial \tilde{X}}{\partial t}\cdot_A \tilde{Z}\right)+\phi_0(\tilde{Y})\left(\tilde{X}\cdot_A \frac{\partial \tilde{Z}}{\partial t}\right)\\
    &\qquad +\phi_0(\tilde{Y})\phi_0\left(\frac{\partial \tilde{X}}{\partial t}\right)\frac{\partial \tilde{Z}}{\partial t}+\phi_0(\tilde{Y})\phi_0(\tilde{X})\frac{\partial^2 \tilde{Z}}{\partial t^2},\\
    (\tilde{X} \:\bar{\cdot}_A^{\phi_0} \:\tilde{Y}) \:\bar{\cdot}_A^{\phi_0} \:\tilde{Z}
    &=\left(\tilde{X}\cdot_A \tilde{Y}+\phi_0(\tilde{X})\frac{\partial \tilde{Y}}{\partial t}\right)\:\bar{\cdot}_A^{\phi_0} \:\tilde{Z}\\
    &=\left(\tilde{X}\cdot_A \tilde{Y}+\phi_0(\tilde{X})\frac{\partial \tilde{Y}}{\partial t}\right)\cdot_A\tilde{Z}\\
    &\qquad +\phi_0\left(\tilde{X}\cdot_A \tilde{Y}+\phi_0(\tilde{X})\frac{\partial \tilde{Y}}{\partial t}\right)\frac{\partial \tilde{Z}}{\partial t}\\
    &=(\tilde{X}\cdot_A \tilde{Y})\cdot_A\tilde{Z}+\left(\phi_0(\tilde{X})\frac{\partial \tilde{Y}}{\partial t}\right)\cdot_A\tilde{Z}\\
    &\qquad +\phi_0(\tilde{X}\cdot_A \tilde{Y})\frac{\partial \tilde{Z}}{\partial t}+\phi_0(\tilde{X})\phi_0\left(\frac{\partial \tilde{Y}}{\partial t}\right)\frac{\partial \tilde{Z}}{\partial t}\\
    &=(\tilde{X}\cdot_A \tilde{Y})\cdot_A\tilde{Z}+\phi_0(\tilde{X})\left(\frac{\partial \tilde{Y}}{\partial t}\cdot_A\tilde{Z}\right)\\
    &\qquad +\phi_0(\tilde{X}\cdot_A \tilde{Y})\frac{\partial \tilde{Z}}{\partial t}+\phi_0(\tilde{X})\phi_0\left(\frac{\partial \tilde{Y}}{\partial t}\right)\frac{\partial \tilde{Z}}{\partial t},\\
    (\tilde{Y} \:\bar{\cdot}_A^{\phi_0} \:\tilde{X}) \:\bar{\cdot}_A^{\phi_0} \:\tilde{Z}
    &=(\tilde{Y}\cdot_A \tilde{X})\cdot_A\tilde{Z}+\phi_0(\tilde{Y})\left(\frac{\partial \tilde{X}}{\partial t}\cdot_A\tilde{Z}\right)\\
    &\qquad +\phi_0(\tilde{Y}\cdot_A \tilde{X})\frac{\partial \tilde{Z}}{\partial t}+\phi_0(\tilde{Y})\phi_0\left(\frac{\partial \tilde{X}}{\partial t}\right)\frac{\partial \tilde{Z}}{\partial t}.
\end{align*}
\begin{align*}
    \therefore\ \tilde{X}& \:\bar{\cdot}_A^{\phi_0} \:(\tilde{Y} \:\bar{\cdot}_A^{\phi_0} \:\tilde{Z})-\tilde{Y} \:\bar{\cdot}_A^{\phi_0} \:(\tilde{X} \:\bar{\cdot}_A^{\phi_0} \:\tilde{Z})-(\tilde{X} \:\bar{\cdot}_A^{\phi_0} \:\tilde{Y}) \:\bar{\cdot}_A^{\phi_0} \:\tilde{Z}+(\tilde{Y} \:\bar{\cdot}_A^{\phi_0} \:\tilde{X}) \:\bar{\cdot}_A^{\phi_0} \:\tilde{Z}\\
    &=\tilde{X} \cdot_A (\tilde{Y}\cdot_A \tilde{Z})-\tilde{Y} \cdot_A(\tilde{X} \cdot_A\tilde{Z})-(\tilde{X} \cdot_A\tilde{Y}) \cdot_A\tilde{Z}+(\tilde{Y} \cdot_A\tilde{X}) \cdot_A\tilde{Z}\\
    &\qquad +(\rho_A(\tilde{X})(\phi_0(\tilde{Y}))-\phi_0(\tilde{X}\cdot_A\tilde{Y}))\frac{\partial \tilde{Z}}{\partial t}\\
    &\qquad -(\rho_A(\tilde{Y})(\phi_0(\tilde{X}))-\phi_0(\tilde{Y}\cdot_A\tilde{X}))\frac{\partial \tilde{Z}}{\partial t}\\
    &=(\delta_A\phi_0)(\tilde{X},\tilde{Y})\frac{\partial \tilde{Z}}{\partial t}-(\delta_A\phi_0)(\tilde{Y},\tilde{X})\frac{\partial \tilde{Z}}{\partial t}.
\end{align*}
Here 
the last step follows from that $\cdot_A$ is a left-symmetric algebra structure on $\Gamma(A)$. Since $\tilde{Z}$ is arbitrary, the triple $(\tilde{A},\hat{\cdot}_A^{\phi_0},\hat{\rho}_A^{\phi_0})$ is a left-symmetric algebroid over $M\times \mathbb{R}$ if and only if $(A,\phi_0)$ is a Jacobi-left-symmetric algebroid over $M$, i.e., $\delta_{A}\phi_0$ is symmetric. The bundle map $\Psi:\tilde{A}\rightarrow \tilde{A}, (v,t)\mapsto (e^t v,t)$, is an isomorphism of vector bundles and
\begin{align*}
\hat{\rho}_A^{\phi_0}\circ \Psi=\bar{\rho}_A^{\phi_0},\qquad 
\Psi(\tilde{X}\:\bar{\cdot}_A^{\phi_0}\:\tilde{Y})=\Psi(\tilde{X})\:\hat{\cdot}_A^{\phi_0}\:\Psi(\tilde{Y}).
\end{align*}
Therefore $(\tilde{A},\hat{\cdot}_A^{\phi_0},\hat{\rho}_A^{\phi_0})$ is a left-symmetric algebroid if and only if $(\tilde{A},\bar{\cdot}_A^{\phi_0},\bar{\rho}_A^{\phi_0})$ is a left-symmetric algebroid. The last claim is obvious by (\ref{hat kakko}), (\ref{hat anchor}), (\ref{bar kakko}), (\ref{bar anchor}), (\ref{hat dot}) and (\ref{bar dot}).
\end{proof}


\begin{theorem}\label{Koszul-Vinbergization}
Let $(A,\phi_0)$ be a Jacobi-left-symmetric algebroid over $M$ and $h$ an element in $\Gamma(S^2A)$. Then $h$ in $\Gamma (S^2A)$ is a Jacobi-Koszul-Vinberg structure on $(A,\phi_0)$ if and only if $\tilde{h}:=e^{-t}h$ in $\Gamma (S^2\tilde{A})$ is a Koszul-Vinberg structure on the left-symmetric algebroid $(\tilde{A},\bar{\cdot}_A^{\phi_0},\bar{\rho}_A^{\phi_0})$. 
\end{theorem}

\begin{proof} 
For any $\tilde{\alpha}, \tilde{\beta}$ and $\tilde{\gamma}$ in $\Gamma(\tilde{A}^*)$,
\begin{align*}
    \llbracket \tilde{h}&,\tilde{h} \rrbracket_{\tilde{A}}(\tilde{\alpha}, \tilde{\beta}, \tilde{\gamma})\\
    &=\bar{\rho}_A^{\phi_0}(\tilde{h}^{\sharp}\tilde{\alpha}) \tilde{h}(\tilde{\beta},\tilde{\gamma})
    -\bar{\rho}_A^{\phi_0}(\tilde{h}^{\sharp}\tilde{\beta}) \tilde{h}(\tilde{\alpha},\tilde{\gamma})
    +\langle \tilde{\alpha}, \tilde{h}^{\sharp}\tilde{\beta} \:\bar{\cdot}_A^{\phi_0} \:\tilde{h}^{\sharp}\tilde{\gamma} \rangle 
    -\langle \tilde{\beta}, \tilde{h}^{\sharp}\tilde{\alpha} \:\bar{\cdot}_A^{\phi_0} \:\tilde{h}^{\sharp}\tilde{\gamma} \rangle \\
    &\qquad -\langle \tilde{\gamma}, [\tilde{h}^{\sharp}\tilde{\alpha}, \tilde{h}^{\sharp}\tilde{\beta} \bar{]}_A^{\phi_0} \rangle\\
    &=\left(\rho_A(e^{-t}h^{\sharp}\tilde{\alpha})+\phi_0(e^{-t}h^{\sharp}\tilde{\alpha})\frac{\partial}{\partial t}\right) \left(e^{-t}h(\tilde{\beta},\tilde{\gamma})\right)\\
    &\qquad -\left(\rho_A(e^{-t}h^{\sharp}\tilde{\beta})+\phi_0(e^{-t}h^{\sharp}\tilde{\beta})\frac{\partial}{\partial t}\right) \left(e^{-t}h(\tilde{\alpha},\tilde{\gamma})\right)\\
    &\qquad +\left\langle \tilde{\alpha}, (e^{-t}h^{\sharp}\tilde{\beta}) \cdot_A(e^{-t}h^{\sharp}\tilde{\gamma})+\phi_0(e^{-t}h^{\sharp}\tilde{\beta})\frac{\partial e^{-t}h^{\sharp}\tilde{\gamma}}{\partial t} \right\rangle\\
    &\qquad -\left\langle \tilde{\beta}, (e^{-t}h^{\sharp}\tilde{\alpha}) \cdot_A(e^{-t}h^{\sharp}\tilde{\gamma})+\phi_0(e^{-t}h^{\sharp}\tilde{\alpha})\frac{\partial e^{-t}h^{\sharp}\tilde{\gamma}}{\partial t} \right\rangle\\
    &\qquad -\left\langle \tilde{\gamma}, [e^{-t}h^{\sharp}\tilde{\alpha}, e^{-t}h^{\sharp}\tilde{\beta} ]_A\right.\\
    &\qquad \qquad \qquad \left. +\phi_0(e^{-t}h^{\sharp}\tilde{\alpha})\frac{\partial e^{-t}h^{\sharp}\tilde{\beta}}{\partial t}-\phi_0(e^{-t}h^{\sharp}\tilde{\beta})\frac{\partial e^{-t}h^{\sharp}\tilde{\alpha}}{\partial t} \right\rangle\\
    &=e^{-2t}\rho_A(h^{\sharp}\tilde{\alpha})\left(h(\tilde{\beta},\tilde{\gamma})\right)\\
    &\qquad +e^{-t}\phi_0(h^{\sharp}\tilde{\alpha})\left(-e^{-t}h(\tilde{\beta},\tilde{\gamma})+e^{-t}h\left(\frac{\partial \tilde{\beta}}{\partial t},\tilde{\gamma}\right)+e^{-t}h\left(\tilde{\beta},\frac{\partial \tilde{\gamma}}{\partial t}\right)\right)\\
    &\qquad -e^{-2t}\rho_A(h^{\sharp}\tilde{\beta})\left(h(\tilde{\alpha},\tilde{\gamma})\right)\\
    &\qquad -e^{-t}\phi_0(h^{\sharp}\tilde{\beta})\left(-e^{-t}h(\tilde{\alpha},\tilde{\gamma})+e^{-t}h\left(\frac{\partial \tilde{\alpha}}{\partial t},\tilde{\gamma}\right)+e^{-t}h\left(\tilde{\alpha},\frac{\partial \tilde{\gamma}}{\partial t}\right)\right)\\
    &\qquad +e^{-2t}\langle \tilde{\alpha}, h^{\sharp}\tilde{\beta}\cdot_A h^{\sharp}\tilde{\gamma}\rangle +e^{-t}\phi_0(h^{\sharp}\tilde{\beta})\left\langle \tilde{\alpha}, -e^{-t}h^{\sharp}\tilde{\gamma}+e^{-t}h^{\sharp}\frac{\partial \tilde{\gamma}}{\partial t}\right\rangle\\
    &\qquad -e^{-2t}\langle \tilde{\beta}, h^{\sharp}\tilde{\alpha}\cdot_A h^{\sharp}\tilde{\gamma}\rangle -e^{-t}\phi_0(h^{\sharp}\tilde{\alpha})\left\langle \tilde{\beta}, -e^{-t}h^{\sharp}\tilde{\gamma}+e^{-t}h^{\sharp}\frac{\partial \tilde{\gamma}}{\partial t}\right\rangle\\
    &\qquad -e^{-2t}\langle \tilde{\gamma}, [h^{\sharp}\tilde{\alpha}, h^{\sharp}\tilde{\beta} ]_A\rangle -e^{-t}\phi_0(h^{\sharp}\tilde{\alpha})\left\langle \tilde{\gamma}, -e^{-t}h^{\sharp}\tilde{\beta}+e^{-t}h^{\sharp}\frac{\partial \tilde{\beta}}{\partial t}\right\rangle\\
    &\qquad +e^{-t}\phi_0(h^{\sharp}\tilde{\beta})\left\langle \tilde{\gamma}, -e^{-t}h^{\sharp}\tilde{\alpha}+e^{-t}h^{\sharp}\frac{\partial \tilde{\alpha}}{\partial t}\right\rangle\\
    &=e^{-2t}\left(\rho_A(h^{\sharp}\tilde{\alpha})\left(h(\tilde{\beta},\tilde{\gamma})\right)-\rho_A(h^{\sharp}\tilde{\beta})\left(h(\tilde{\alpha},\tilde{\gamma})\right)\right.\\
    &\qquad \qquad  +\langle \tilde{\alpha}, h^{\sharp}\tilde{\beta}\cdot_A h^{\sharp}\tilde{\gamma}\rangle-\langle \tilde{\beta}, h^{\sharp}\tilde{\alpha}\cdot_A h^{\sharp}\tilde{\gamma}\rangle-\langle \tilde{\gamma}, [h^{\sharp}\tilde{\alpha}, h^{\sharp}\tilde{\beta} ]_A\rangle \\
    &\qquad \qquad \left. 
    +h(\phi_0,\tilde{\alpha})h(\tilde{\beta},\tilde{\gamma})-h(\phi_0,\tilde{\beta})h(\tilde{\gamma},\tilde{\alpha}) \right)\\
    &=e^{-2t}(\llbracket h,h\rrbracket_A(\tilde{\alpha},\tilde{\beta},\tilde{\gamma})+h(\phi_0,\tilde{\alpha})h(\tilde{\beta},\tilde{\gamma})-h(\phi_0,\tilde{\beta})h(\tilde{\gamma},\tilde{\alpha}))\\
    &=e^{-2t}\llbracket h,h \rrbracket_A^{\phi_0}(\tilde{\alpha},\tilde{\beta},\tilde{\gamma}).
\end{align*}
Therefore $h$ is a Jacobi-Koszul-Vinberg structure on $(A,\phi_0)$ if and only if $\tilde{h}$ is a Koszul-Vinberg structure on $(\tilde{A},\bar{\cdot}_A^{\phi_0},\bar{\rho}_A^{\phi_0})$.
\end{proof}

\begin{defn}
$\tilde{h}$ in Theorem \ref{Koszul-Vinbergization} is called the {\it Koszul-Vinbergization} of a Jacobi-Koszul-Vinberg structure $h$ on $(A,\phi_0)$.    
\end{defn}

Now we define a Jacobi-Koszul-Vinberg manifold as a symmetric analogue of a Jacobi manifold.
\begin{defn}\label{JKV}
A {\it Jacobi-Koszul-Vinberg structure} on an affine manifold $(M, \nabla)$ is 
a pair $(h ,E)$ in $\Gamma (S^{2}TM)\times \mathfrak{X}(M)$ satisfying, \\
for any $\alpha,\beta,\gamma$ in $\Omega^1(M)$,
\begin{align*}
({\rm i}) \; &(\nabla_{h^{\sharp}\alpha} h)(\beta,\gamma)-\langle \alpha, E \rangle h(\beta,\gamma) 
=(\nabla_{h^{\sharp}\beta} h)(\alpha,\gamma)-\langle \beta, E \rangle h(\alpha,\gamma), \\
({\rm ii}) \; &(\nabla_{E} h)(\beta,\gamma)-\langle \gamma, \nabla_{h^{\sharp}\beta} E \rangle
+\langle \beta, E \rangle \langle \gamma, E \rangle=0, \\ 
({\rm iii}) \; &\nabla_E E=0. 
\end{align*}
$(M, \nabla, h ,E)$ is called a {\it Jacobi-Koszul-Vinberg manifold}.
\end{defn}

Let $(M,\nabla)$ be an affine manifold.  As previously stated in Example \ref{barnabla}, $(TM\oplus \mathbb{R},\bar{\nabla},\mbox{pr}_1)$ is a left-symmetric algebroid over $M$.  Furthermore, because of Example \ref{01cocy} and Proposition \ref{cocyc}, $((TM\oplus \mathbb{R},\bar{\nabla},\mbox{pr}_1),(0,1))$ is a Jacobi-left-symmetric algebroid over $M$.

\begin{prop}
Let $(M, \nabla)$ be an affine manifold.  
Identify $(h ,E)$ in $\Gamma (S^{2}TM)\times \mathfrak{X}(M)$ with $H$ in $\Gamma (S^{2}(TM\oplus \mathbb{R}))$ by 
$\displaystyle H = h + \frac{\partial}{\partial t} \otimes E+ E\otimes \frac{\partial}{\partial t}$. Under this identification, the followings are equivalent; 
\begin{align*}
&(1)\; (h ,E) \;\text{is a Jacobi-Koszul-Vinberg structure on}\; (M, \nabla) \\
&(2)\; \llbracket H,H \rrbracket_{TM\oplus \mathbb{R}}^{(0,1)}=0.
\end{align*}
\end{prop}
\begin{proof}
Since $(\alpha, f)$ in $\Omega^1(M)\times C^{\infty}(M)$ is identified with $\alpha+fdt$ in $\Gamma((TM\oplus \mathbb{R})^*)$, 
the followings hold:
\begin{align*}
H((\alpha_1, f_1),(\alpha_2, f_2))&=h(\alpha_1,\alpha_2)+\langle \alpha_1,E \rangle f_2+\langle \alpha_2,E \rangle f_1, \\
H^{\sharp}(\alpha_1, f_1)&=(h^{\sharp}\alpha_1 +f_1 E, \langle \alpha_1, E\rangle), \\
H^{\sharp}(\alpha_1, f_1)\cdot_{TM\oplus \mathbb{R}} H^{\sharp}(\alpha_2, f_2)
&=\bar{\nabla}_{H^{\sharp}(\alpha_1, f_1)}H^{\sharp}(\alpha_2, f_2) \\
&=(\nabla_{h^{\sharp}\alpha_1 +f_1 E}(h^{\sharp}\alpha_2 +f_2 E), (h^{\sharp}\alpha_1 +f_1 E)\langle \alpha_2,E \rangle),
\end{align*}
for $(\alpha_i, f_i)$ in $\Omega^1(M)\times C^{\infty}(M) \quad (i=1,2)$.  By using these equalities, it follows that for $(\alpha_i, f_i)$ in $\Omega^1(M)\times C^{\infty}(M) \quad (i=1,2,3)$,
\begin{align*}
&\llbracket H,H \rrbracket_{TM\oplus \mathbb{R}} 
((\alpha_1, f_1),(\alpha_2, f_2),(\alpha_3, f_3)) \\
=&\;(\mbox{pr}_1\circ H^{\sharp}(\alpha_1, f_1)) H((\alpha_2, f_2),(\alpha_3, f_3)) \\
&-(\mbox{pr}_1\circ H^{\sharp}(\alpha_2, f_2)) H((\alpha_1, f_1),(\alpha_3, f_3)) \\
&-\langle (\alpha_2, f_2),\: H^{\sharp}(\alpha_1, f_1)\cdot_{TM\oplus \mathbb{R}}H^{\sharp}(\alpha_3, f_3)\rangle \\
&+\langle (\alpha_1, f_1),\: H^{\sharp}(\alpha_2, f_2)\cdot_{TM\oplus \mathbb{R}}H^{\sharp}(\alpha_3, f_3)\rangle \\
&-\langle (\alpha_3, f_3),\: [H^{\sharp}(\alpha_1, f_1), H^{\sharp}(\alpha_2, f_2)]_{TM\oplus \mathbb{R}} \rangle \\
=&\; (h^{\sharp}\alpha_1 +f_1 E)(h(\alpha_2,\alpha_3)+\langle \alpha_2,E \rangle f_3+\langle \alpha_3,E \rangle f_2) \\
&-(h^{\sharp}\alpha_2 +f_2 E)(h(\alpha_1,\alpha_3)+\langle \alpha_1,E \rangle f_3+\langle \alpha_3,E \rangle f_1) \\
&-(\langle \alpha_2, \:\nabla_{h^{\sharp}\alpha_1 +f_1 E}(h^{\sharp}\alpha_3 +f_3 E)\rangle +f_2(h^{\sharp}\alpha_1 +f_1 E)\langle \alpha_3, E\rangle) \\
&+(\langle \alpha_1, \:\nabla_{h^{\sharp}\alpha_2 +f_2 E}(h^{\sharp}\alpha_3 +f_3 E)\rangle +f_1(h^{\sharp}\alpha_2 +f_2 E)\langle \alpha_3, E\rangle) \\
&-(\langle \alpha_3, \:[h^{\sharp}\alpha_1 +f_1 E, \:h^{\sharp}\alpha_2 +f_2 E]\rangle \\
&\quad\quad +f_3\{(h^{\sharp}\alpha_1 +f_1 E)\langle \alpha_2, E\rangle-(h^{\sharp}\alpha_2 +f_2 E)\langle \alpha_1, E\rangle\}) \\
=&\; (h^{\sharp}\alpha_1 +f_1 E)h(\alpha_2,\alpha_3) 
+\langle \alpha_2, E\rangle (h^{\sharp}\alpha_1 +f_1 E)f_3
+\langle \alpha_3, E\rangle (h^{\sharp}\alpha_1 +f_1 E)f_2 \\
&-(h^{\sharp}\alpha_2 +f_2 E)h(\alpha_1,\alpha_3) 
-\langle \alpha_1, E\rangle (h^{\sharp}\alpha_2 +f_2 E)f_3
-\langle \alpha_3, E\rangle (h^{\sharp}\alpha_2 +f_2 E)f_1 \\
&-\langle \alpha_2, \:\nabla_{h^{\sharp}\alpha_1 +f_1 E}(h^{\sharp}\alpha_3 +f_3 E)\rangle
+\langle \alpha_1, \:\nabla_{h^{\sharp}\alpha_2 +f_2 E}(h^{\sharp}\alpha_3 +f_3 E)\rangle \\
&-\langle \alpha_3, \:[h^{\sharp}\alpha_1 +f_1 E, \:h^{\sharp}\alpha_2 +f_2 E]\rangle \\
=&\;(h^{\sharp}\alpha_1 +f_1 E)h(\alpha_2,\alpha_3) 
+\langle \alpha_3, E\rangle (h^{\sharp}\alpha_1 +f_1 E)f_2 \\
&-(h^{\sharp}\alpha_2 +f_2 E)h(\alpha_1,\alpha_3) 
-\langle \alpha_3, E\rangle (h^{\sharp}\alpha_2 +f_2 E)f_1 \\
&-\langle \alpha_2, \:\nabla_{h^{\sharp}\alpha_1 +f_1 E}\:h^{\sharp}\alpha_3\rangle -f_3 \langle \alpha_2, \:\nabla_{h^{\sharp}\alpha_1 +f_1 E}\:E\rangle \\
&+\langle \alpha_1, \:\nabla_{h^{\sharp}\alpha_2 +f_2 E}\:h^{\sharp}\alpha_3\rangle +f_3 \langle \alpha_1, \:\nabla_{h^{\sharp}\alpha_2 +f_2 E}\:E\rangle \\
&-\langle \alpha_3, \:\nabla_{h^{\sharp}\alpha_1 +f_1 E}(h^{\sharp}\alpha_2 +f_2 E)-\nabla_{h^{\sharp}\alpha_2 +f_2 E}(h^{\sharp}\alpha_1 +f_1 E)\rangle \\
=&\;(h^{\sharp}\alpha_1 +f_1 E)h(\alpha_2,\alpha_3)
-(h^{\sharp}\alpha_2 +f_2 E)h(\alpha_1,\alpha_3) \\ 
&-\langle \alpha_2, \:\nabla_{h^{\sharp}\alpha_1 +f_1 E}\:h^{\sharp}\alpha_3\rangle -f_3 \langle \alpha_2, \:\nabla_{h^{\sharp}\alpha_1 +f_1 E}\:E\rangle \\
&+\langle \alpha_1, \:\nabla_{h^{\sharp}\alpha_2 +f_2 E}\:h^{\sharp}\alpha_3\rangle +f_3 \langle \alpha_1, \:\nabla_{h^{\sharp}\alpha_2 +f_2 E}\:E\rangle \\
&-\langle \alpha_3, \:\nabla_{h^{\sharp}\alpha_1 +f_1 E}\:h^{\sharp}\alpha_2\rangle -f_2 \langle \alpha_3, \:\nabla_{h^{\sharp}\alpha_1 +f_1 E}\:E\rangle \\
&+\langle \alpha_3, \:\nabla_{h^{\sharp}\alpha_2 +f_2 E}\:h^{\sharp}\alpha_1\rangle +f_1 \langle \alpha_3, \:\nabla_{h^{\sharp}\alpha_2 +f_2 E}\:E\rangle \\
=&-(\nabla_{h^{\sharp}\alpha_1 +f_1 E}\:h)(\alpha_2,\alpha_3)
+(\nabla_{h^{\sharp}\alpha_2 +f_2 E}\:h)(\alpha_1,\alpha_3) \\
&-f_3 \langle \alpha_2, \:\nabla_{h^{\sharp}\alpha_1 +f_1 E}\:E\rangle +f_3 \langle \alpha_1, \:\nabla_{h^{\sharp}\alpha_2 +f_2 E}\:E\rangle \\
&-f_2 \langle \alpha_3, \:\nabla_{h^{\sharp}\alpha_1 +f_1 E}\:E\rangle +f_1 \langle \alpha_3, \:\nabla_{h^{\sharp}\alpha_2 +f_2 E}\:E\rangle \\
=&-(\nabla_{h^{\sharp}\alpha_1}h)(\alpha_2,\alpha_3)
+(\nabla_{h^{\sharp}\alpha_2}h)(\alpha_1,\alpha_3) \\
&-f_1((\nabla_E h)(\alpha_2,\alpha_3)-\langle \alpha_3, \:\nabla_{h^{\sharp}\alpha_2}E\rangle) \\
&+f_2((\nabla_E h)(\alpha_1,\alpha_3)-\langle \alpha_3, \:\nabla_{h^{\sharp}\alpha_1}E\rangle) \\
&+f_3(\langle \alpha_1, \:\nabla_{h^{\sharp}\alpha_2}E\rangle
-\langle \alpha_2, \:\nabla_{h^{\sharp}\alpha_1}E\rangle
+f_2\langle \alpha_1, \:\nabla_E E\rangle
-f_1\langle \alpha_2, \:\nabla_E E\rangle). 
\end{align*}
Hence
\begin{align*}
&\llbracket H,H \rrbracket_{TM\oplus \mathbb{R}}^{(0,1)} ((\alpha_1, f_1),(\alpha_2, f_2),(\alpha_3, f_3))  \\
=&\: \llbracket H,H \rrbracket_{TM\oplus \mathbb{R}} 
((\alpha_1, f_1),(\alpha_2, f_2),(\alpha_3, f_3)) \\
&+H((0,1),(\alpha_1, f_1))H((\alpha_2, f_2),(\alpha_3, f_3)) \\
&-H((0,1),(\alpha_2, f_2))H((\alpha_1, f_1),(\alpha_3, f_3)) \\
=&\: \llbracket H,H \rrbracket_{TM\oplus \mathbb{R}} 
((\alpha_1, f_1),(\alpha_2, f_2),(\alpha_3, f_3)) \\
&+\langle \alpha_1,E\rangle(h(\alpha_2,\alpha_3)+\langle \alpha_2,E \rangle f_3+\langle \alpha_3,E \rangle f_2) \\
&-\langle \alpha_2,E\rangle(h(\alpha_1,\alpha_3)+\langle \alpha_1,E \rangle f_3+\langle \alpha_3,E \rangle f_1) \\
=&\: \llbracket H,H \rrbracket_{TM\oplus \mathbb{R}} 
((\alpha_1, f_1),(\alpha_2, f_2),(\alpha_3, f_3)) \\
&+\langle \alpha_1,E\rangle(h(\alpha_2,\alpha_3)
+\langle \alpha_3,E \rangle f_2) 
-\langle\alpha_2,E\rangle(h(\alpha_1,\alpha_3)
+\langle \alpha_3,E \rangle f_1) \\
=&-(\nabla_{h^{\sharp}\alpha_1}h)(\alpha_2,\alpha_3)
+\langle \alpha_1,E\rangle h(\alpha_2,\alpha_3) \\
&+(\nabla_{h^{\sharp}\alpha_2}h)(\alpha_1,\alpha_3) 
-\langle\alpha_2,E\rangle h(\alpha_1,\alpha_3) \\
&-f_1((\nabla_E h)(\alpha_2,\alpha_3)
-\langle \alpha_3, \:\nabla_{h^{\sharp}\alpha_2}E\rangle 
+\langle\alpha_2,E\rangle \langle\alpha_3,E\rangle) \\
&+f_2((\nabla_E h)(\alpha_1,\alpha_3)
-\langle \alpha_3, \:\nabla_{h^{\sharp}\alpha_1}E\rangle
+\langle\alpha_1,E\rangle \langle\alpha_3,E\rangle) \\
&+f_3(\langle \alpha_1, \:\nabla_{h^{\sharp}\alpha_2}E\rangle
-\langle \alpha_2, \:\nabla_{h^{\sharp}\alpha_1}E\rangle
+f_2\langle \alpha_1, \:\nabla_E E\rangle
-f_1\langle \alpha_2, \:\nabla_E E\rangle). 
\end{align*}
Therefore, this vanishes for any $\alpha_i$ in $\Omega^1(M)$ and any $f_i$ in $C^{\infty}(M) \quad (i=1,2,3)$ if and only if (i)-(iii) in Definition \ref{JKV} hold 
for any $\alpha,\beta,\gamma$ in $\Omega^1(M)$.
\end{proof}

Let $(M,g)$ be a (pesudo-)Riemannian manifold, $\nabla$ a torsion-free connection on $M$ and $\theta$ a 1-form on $M$.  The quadruple $(M,\nabla,g,\theta)$ is called a {\it semi-Weyl manifold} \cite{Mat} if 
$\nabla g+\theta \otimes g$ is symmetric.  
A semi-Weyl manifold $(M,\nabla,g,\theta)$ is called a {\it locally conformally Hessian manifold} \cite{Osi} if $\nabla$ is flat and $\theta$ is $d$-closed in addition. 
\begin{prop}
Let $(M, \nabla, h ,E)$ be a Jacobi-Koszul-Vinberg manifold. Assume $h$ in $\Gamma (S^{2}TM)$ is non-degenerate.  
Then $(M, \nabla, g ,\theta)$ is a locally conformally Hessian manifold, where $g$ is the corresponding metric to $h$ and $\theta = g^{\flat} E$.
\end{prop}

\begin{proof}
For $\alpha,\beta,\gamma$ in $\Omega^1(M)$, 
set $X=h^{\sharp}\alpha,\: Y=h^{\sharp}\beta,\: Z=h^{\sharp}\gamma$.  Then, by a straightforward calculation, the following equalities hold:
\begin{align*}
&(\nabla_{h^{\sharp}\alpha} h)(\beta,\gamma)-(\nabla_{h^{\sharp}\beta} h)(\alpha,\gamma)
=-(\nabla_X g)(Y,Z)+(\nabla_Y g)(X,Z), \\
&(\nabla_{h^{\sharp}\alpha} h)(\beta,\gamma)-\langle \alpha, E \rangle h(\beta,\gamma) 
-(\nabla_{h^{\sharp}\beta} h)(\alpha,\gamma)+\langle \beta, E \rangle h(\alpha,\gamma) \\
=&-(\nabla_X g)(Y,Z)-\theta(X)g(Y,Z)
+(\nabla_Y g)(X,Z)+\theta(Y)g(X,Z).
\end{align*}
This means that if $h$ is non-degenerate, (i) in Definition \ref{JKV} is equivalent to that $(M, \nabla, g ,\theta)$ is a semi-Weyl manifold.  Furthermore,
\begin{align*}
(d\theta)(X,Y)&=X(\theta(Y))-Y(\theta(X))-\theta([X,Y]) \\
&=X(g(Y,E))-Y(g(X,E))-g([X,Y],E) \\
&=X(g(Y,E))-Y(g(X,E))-g(\nabla_X Y-\nabla_Y X,E) \\
&=(\nabla_X g)(Y,E)+g(Y,\nabla_X E)-(\nabla_Y g)(X,E)-g(X,\nabla_Y E) \\
&=-\theta (X)g(Y,E)+\theta (Y)g(X,E)+g(Y,\nabla_X E)-g(X,\nabla_Y E) \\
&=g(Y,\nabla_X E)-g(X,\nabla_Y E). 
\end{align*}
From (ii) in Definition \ref{JKV}, 
$\langle \beta, \nabla_{h^{\sharp}\alpha} E \rangle$ is symmetric for any $\alpha,\beta$ in $\Omega^1(M)$.  
This means that $g(Y,\nabla_X E)$ is symmetric for any $X,Y$ in $\mathfrak{X}(M)$ since 
\begin{equation*}
\langle \beta, \nabla_{h^{\sharp}\alpha} E \rangle =g(Y,\nabla_X E).
\end{equation*}
Therefore, for any $X,Y$ in $\mathfrak{X}(M)$,
\begin{equation*}
(d\theta)(X,Y) =0.
\end{equation*}
\end{proof}


\bigskip

\address{
Department of Mathematics \\
Faculty of Science Division II \\ 
Tokyo University of Science\\ 
1-3 Kagurazaka, Shinjuku-ku\\
Tokyo 162-8601\\
Japan
}
{n-kimura@rs.tus.ac.jp}
\address{
Academic Support Center \\
Kogakuin University\\
2665-1 Nakano-cho, Hachioji\\
Tokyo 192-0015\\
Japan
}
{kt13676@ns.kogakuin.ac.jp}


\begin{thebibliography}{99}

\bibitem{ABB1}
\textsc{A. Abouqateb, M. Boucetta and C. Bourzik},
Contravariant pseudo-Hessian manifolds and their associated Poisson structures,
Differential Geom. Appl. 70 (2020), 101630, 26 pp.

\bibitem{ABB2}
\textsc{A. Abouqateb, M. Boucetta and C. Bourzik},
Submanifolds in Koszul-Vinberg geometry,
Results Math. 77 (2022), no.1, Paper No. 19, 24 pp.


\bibitem{BB2}
\textsc{S. Benayadi and M. Boucetta},
On para-K\"{a}hler Lie algebroids and contravariant pseudo-Hessian structures,
Math. Nachr. 292 (2019), no.7, 1418--1443.

\bibitem{Boy1}
\textsc{M. Nguiffo Boyom},
Cohomology of Koszul-Vinberg algebroids and Poisson manifolds. I, 
Lie algebroids and related topics in differential geometry. 
Banach Center Publ., 54 (2001), 99--110.


\bibitem{LSB1}
\textsc{J. Liu, Y. Sheng and C.Bai},
Left-symmetric bialgebroids and their corresponding Manin triples, 
Differential Geom. Appl. 59 (2018), 91--111.


\bibitem{Mat}
\textsc{H. Matsuzoe},
Geometry of semi-Weyl manifolds and Weyl manifolds,
Kyushu J. Math. 55 (2001), no.1, 107--117.

\bibitem{Osi}
\textsc{P. Osipov},
Locally conformally Hessian and statistical manifolds,
J. Geom. Phys. 193 (2023), Paper No. 104989, 16 pp.


\bibitem{Shi}
\textsc{H. Shima},
The geometry of Hessian structures.
World Scientific Publishing Co. Pte. Ltd., Hackensack, NJ, 2007. 


\bibitem{WLS}
{Q. Wang, J. Liu and Y. Sheng},
Koszul-Vinberg structures and compatible structures on left-symmetric algebroids,
Int. J. Geom. Methods Mod. Phys. 17 (2020), no.13, 2050199, 28 pp.

\end{thebibliography}
\end{document}